\documentclass[a4paper]{article}
\usepackage{amsmath, amsthm, amssymb, fullpage, enumerate, xifthen, xspace, url}

\newtheorem{proposition}{Proposition}[section]
\newtheorem{thm}[proposition]{Theorem}
\newtheorem{lem}[proposition]{Lemma}

\theoremstyle{remark}
\newtheorem*{rem}{Remark}
\theoremstyle{definition}
\newtheorem{problem}[proposition]{Problem}

\newcommand*{\z}[1][Z]{\mathbb{\uppercase{#1}}}
\newcommand*{\zp}[1][Z]{\mathbb{\uppercase{#1}}_{+}}

\newcommand{\sg}{\sigma}
\newcommand{\eps}{\varepsilon}
\newcommand{\mc}{\sg_{\mathrm{crit}}}
\newcommand*{\prob}[2][]{\mathbb{P}_{#1}(#2)}
\newcommand*{\probb}[2][]{\mathbb{P}_{#1}\bigl(#2\bigr)}
\newcommand*{\probbb}[2][]{\mathbb{P}_{#1}\Bigl(#2\Bigr)}
\newcommand*{\mean}[2][]{\mathbb{E}_{#1}(#2)}
\newcommand*{\Po}[1]{\operatorname{Po}(#1)}
\newcommand*{\hier}[3]{#1^{(#2)}_{#3}}
\newcommand*{\ee}[1][]{\ifthenelse{\isempty{#1}}{\mathrm{e}}{\mathrm{e}^{#1}}}
\newcommand*{\abs}[1]{\lvert#1\rvert}
\newcommand*{\cprob}[3][]{\mathbb{P}_{#1}(#2\mid #3)}

\newcommand{\var}{\operatorname{Var}}
\newcommand{\wto}{\stackrel{\mathrm{w}}{\to}}
\newcommand*{\seq}[3][1]{#2_{#1},\ldots,#2_{#3}}
\newcommand*{\set}[3][1]{\{\seq[#1]{#2}{#3}\}}
\newcommand*{\floor}[1]{\lfloor#1\rfloor}
\newcommand*{\ceil}[1]{\lceil#1\rceil}
\newcommand*{\gnl}[1][n]{G^\bullet_{#1}(\la)}
\newcommand*{\gnm}[2][n]{G^\bullet_{#1}(#2)}
\newcommand*{\bfrac}[2]{\genfrac{(}{)}{}{}{#1}{#2}}

\newcommand{\Lr}[1]{Lemma~\ref{#1}}
\newcommand{\Tr}[1]{Theorem~\ref{#1}}
\newcommand{\Sr}[1]{Section~\ref{#1}}
\newcommand{\Prr}[1]{Pro\-position~\ref{#1}}
\newcommand{\Prb}[1]{Problem~\ref{#1}}
\newcommand{\defi}[1]{{\emph{#1}}}

\newcommand{\ie}{i.e.\ }
\newcommand{\eg}{e.g.\ }
\newcommand{\fe}{for every}
\newcommand{\ER}{Erd\H os--R\'enyi}
\newcommand{\LRP}{long range percolation}
\newcommand{\la}{\lambda}
\newcommand{\grp}[1][2]{\bigoplus_{i \in \z[n]} \z_{#1}}
\newcommand{\grpb}{\grp[b]}

\newcommand{\pem}{Poisson edge model\xspace}
\newcommand{\ppm}{Poisson particle model\xspace}

\usepackage{authblk}

\title{Percolation on an infinitely generated group}

\author{Agelos Georgakopoulos}
\author{John Haslegrave}
\affil{{Mathematics Institute}\\
	{University of Warwick}\\
	{CV4 7AL, UK}\thanks{Supported by the European Research Council (ERC) under the European Union's Horizon 2020 research and innovation programme (grant agreement no.\ 639046).}\\}
\begin{document}
\maketitle 

\begin{abstract}
We give an example of a long range Bernoulli percolation process on a group non-quasi-isometric with $\z$, in which clusters are almost surely finite for all values of the parameter. This random graph admits diverse equivalent definitions, and we study their ramifications. We also study its expected size and point out certain phase transitions.
\end{abstract}

\section{Introduction}

We consider an instance of (long range) Bernoulli percolation on the group $\grpb$, providing the first example of Bernoulli percolation that is subcritical for every value of the parameter on a group non quasi-isometric with $\z$. We observe that this random graph arises in other contexts, and point out further interesting properties. 
\medskip

Most random graph models studied enjoy some form of invariance. For example, the distribution of the \ER\ random graph on the vertex set $[n]$ is invariant under permutations of $[n]$, and this is also true in much more general models, see \eg \cite{LoSzeRan}. Percolation theory provides further examples where a random graph is invariant under the action of some group, \eg $\z^d$, on the set of vertices. In a random geometric graph in the sense of \cite{PenRgg}, the group does not act on the vertex set directly, but on an ambient space in which the vertices live. These models display a `spatial invariance', but there are also examples of `temporal invariance': any random (regular) rooted graph arising as a limit of a sequence of finite graphs in the sense of Benjamini \& Schramm \cite{BeSchrRec} is invariant under taking a step of simple random walk from the root and then declaring the destination to be the root. 
Dynamic percolation \cite{SteiSur} serves as an example of a model with both spatial and temporal invariance. The random graphs introduced in this paper also enjoy both spatial and temporal invariance, but we need to use two seemingly unrelated -- and a-posteriori equivalent -- definitions to see this. 

The `spatial' definition is via percolation on $\Gamma=\grpb$, \ie the direct sum of infinitely many copies of the cyclic group with $b$ elements: we join each pair $x,y\in \Gamma$ with a random number of (parallel) edges with distribution $\Po{\la b^{1-2h(x,y)}}$, where $\la\in \zp[r]$ is the parameter of the model (proportional to the average degree of a vertex), $h(x,y)$ is the first coordinate at which $x,y$ differ, and $\Po{\mu}$ denotes the Poisson distribution with mean $\mu$. The reader is not yet expected to appreciate why $\la b^{1-2h(x,y)}$ was the right choice; for the time being we just note that this random (multi-)graph\footnote{A \defi{multi-graph} is a graph in which we can have several `\defi{parallel}' edges between the same pair of vertices.} is invariant under the natural action of $\Gamma$ on itself, and apart from that it is hard to say anything about it. 

The `temporal' definition of this model is given by the following proposition.

\begin{proposition} \label{thminvariant}
For every $\la\in \zp[r]$, there is a unique rooted connected random multi-graph $(G(\la),o)$ with the root having finite expected degree which is invariant under the following operation.
\begin{equation}\label{star}
\begin{minipage}[c]{0.9\linewidth}{Replace each vertex $v$ of $G(\la)$ (including the root $o$) by $b$ vertices $v_0,\ldots,v_{b-1}$, and for each $i\neq j\in \z_b$ join $v_i$ to $v_j$ with a random number of edges with distribution $\Po{\la/b}$. \\ 
Moreover, replace each edge $uv$ of $G(\la)$, with one of the $b^2$ edges $v_i u_j: i,j\in \z_2$, chosen uniformly at random. All these random experiments are made independently from each other.\\
Choose the root of the resulting graph to be each of $o_0,\ldots,o_{b-1}$ with probability $\frac1{b}$.\\
Finally, if the graph is now disconnected, discard all components not containing the root.}
\end{minipage}
\end{equation}
\end{proposition}

It turns out that $(G(\la),o)$ has the same distribution as the component (a.k.a.\ cluster) of the origin in the aforementioned percolation model on $\Gamma$. (In fact, it is possible to obtain a statement similar to \Prr{thminvariant} when all components are retained, and the corresponding disconnected random graph has the same distribution as our percolation on all of $\Gamma$.) The fact that these two random multi-graphs coincide is far from clear at first sight; we prove this by showing (in \Sr{proofthm}) that they both coincide with a third random graph. The vertices of the latter random graph are the leaves of an infinite tree $T_\infty$ (the \defi{canopy tree}, defined in \Sr{treenotation}). Following the general construction of \defi{Group Walk Random Graphs} (GWRGs) \cite{gwrg}, the choice of which pairs to connect with an edge is made using an experiment involving random walks on $T_\infty$. Thus the choice of the coefficients $\la b^{1-2h(x,y)}$ above was dictated by the behaviour of random walk. This choice is also unique in that it makes the two aforementioned models coincide, and `critical' in a sense explained in \Sr{secintromu}. 

In fact this `third' definition (given in detail in \Sr{secppmpem}) was the starting point of our work. The general motivation is that GWRGs link groups to geometric random graphs, and studying the interactions could be fruitful; we refer the interested reader to  \cite{gwrg} for more details on the background of this construction.
\medskip

Despite having several equivalent definitions, it is very hard to say anything about the structure of $G(\la)$. It is not even obvious whether it is finite or infinite, but our first main result (proved in \Sr{size}) implies that it is almost surely finite:
\begin{thm} \label{Ebounds}
The expected number of vertices $\chi(\la)$ of $G(\la)$ satisfies
\[ \ee[c\la] < \chi(\la) < \ee[\ee^{C\la}]\] 
for some constants $c,C>0$. 
\end{thm}

These bounds leave an enormous gap, but it seems to be very hard to improve them significantly. Computer simulations we performed for $b=2$ and $\la\leq 12$ suggest that $\chi(\la)$ might be of order $\la^{c\la}$. Conjecturing that $\chi(\la) \sim \la^{c\la}$ led us to wonder whether $\chi$ is a continuous/smooth function of $\la$. By adapting a well-known technique of Kesten \cite{Ke81}, the first author and C. Panagiotis \cite{analyticity} proved that $\chi(\la)$ is an analytic function at every $\la \in \zp[r]$, and this statement holds in the full generality of all Bernoulli long or short range percolation models on groups. 

\medskip
Thus $G(\la)$ displays no phase transitions, at least as far as $\chi$ is concerned. Still, we observed some rougher phase transition phenomena. We consider finite versions of our percolation model on $\grpb$ obtained, roughly speaking, by restriction to finite subgroups, and determine the threshold value of $\la$ for obtaining a connected graph. Alternatively, these finite  versions of our model can be described by replacing the canopy tree in our other definition by the full binary tree of depth $n\in \mathbb{N}$.
We prove that there is a phase transition for connectedness, occurring at a sharp threshold $\la_{\mathrm{conn}}$ logarithmic in the size of the graph, while the transition occurs in a window of width proportional to the logarithm of $\la_{\mathrm{conn}}$ (\Sr{secthresholds}).
We remark that this restriction on finite subgroups of $\grpb$ is somewhat related to percolation on Hamming hypercubes (which are Cayley graphs of such subgroups), which has attracted a lot of interest recently, see \cite{HofNachUnl,HofNachHyp} and references therein.

\subsection{Percolation on groups} \label{secintromu}
A well-known conjecture of Benjamini \& Schramm \cite{BeSchrPer}, recently proved by  Duminil-Copin, Goswami, Raoufi, Severo, and Yadin \cite{DCGRSY}, states that $p_{\mathrm c}< 1$ holds for Bernoulli percolation on  every Cayley graph of a group which is not a finite extension of $\z$. 
Our result that $G(\la)$ is almost surely finite \fe\ $\la$ means that the analogue of this statement for \LRP\ on infinitely generated groups is false. 
To explain what we mean by \LRP, let $\mu$ be a probability measure on a (countable) group $\Gamma$. We say that $\mu$ is a \defi{generating measure} of $\Gamma$, if the support of $\mu$ generates $\Gamma$, and $\mu$ is symmetric, \ie $\mu(g)= \mu(g^{-1})$ \fe\ $g\in \Gamma$. 

Every generating measure $\mu$ naturally defines a percolation process on $\Gamma$ as follows. Given $\la \in \zp[r]$, we define a random (multi-)graph $\Gamma_\mu(\la)$ with vertex set $\Gamma$, by letting the number of (parallel) edges between two elements $g,h\in \Gamma$ be an independent Poisson random variable with mean $\la \mu(g^{-1}h)$ (we may as well remove any parallel edges to obtain a simple graph). Such models were already considered \eg in \cite{AizNewTre}.

Note that  $\Gamma_\mu(\la)$  is a $\Gamma$-invariant percolation model, \ie the natural action of $\Gamma$ on $\Gamma_\mu(\la)$ defined by multiplication from the left preserves the probability distribution of $\Gamma_\mu(\la)$. 

Similarly to the standard percolation threshold $p_{\mathrm c}$, we define 
\[\la_{\mathrm c}= \la_{\mathrm c}(\mu) := \sup \{ \la \mid \prob{\Gamma_\mu(\la) \text{ has an infinite component}} = 0 \}.\] 
We remark that $\la_{\mathrm c}$ may be infinite, as is the case with our $G(\la)$. Another  result of this paper implies however that $\la_{\mathrm c}<\infty$ for other choices of $\mu$ on the same group  $\Gamma= \grpb$: suppose $\mu(g)$ is proportional to $\alpha^{-\ell(g)}$, where $\ell(g)$ denotes the index of the last non-zero coordinate of $g\in \grpb$. Then for $\alpha=b^2$ we obtain $G(\la)$ as the component of the origin by the definitions. Moreover, we prove that this $\alpha$ is `critical' in the sense that percolation does occur -- for large enough $\la$ -- for any $\alpha\in (b,b^2)$, but not for $\alpha\geq b^2$ (\Sr{secperco}).

It is interesting to compare this fact with \LRP\ on $\Gamma:= \z$, which is the most studied instance of this model.  Let $\mu(i) = c\abs i^{-s}$, with $s\in (1,\infty)$, and $c$ a suitable normalising constant that matters little. 
It has been proved that for $s>2$ we have no percolation, \ie $\la_{\mathrm c}= \infty$ \cite{SchulLong}, while for $s\in (1,2]$ we have  $\la_{\mathrm c}< \infty$ \cite{NewSchulOne}. The case $s=2$ is of particular interest, and is considered as the `critical' case. Indeed, when $s=2$, the percolation density is discontinuous at $\la_{\mathrm c}$ \cite{AizNewDis}. 

Interestingly, this case is related to our critical case $\alpha=b^2$ in Example~3 above. Indeed, if we enumerate the elements of $\grpb$ appropriately, namely by thinking of the elements of $\grpb$ as natural numbers expressed in base $b$, then for pairs of `numbers' $x,y$ far apart, the probability to join $x$ to $y$ with an edge decays like $\abs{x-y}^{-2}$ in both models. (But when $\abs{x-y}$ is small, then in our example, this probability can be much smaller than the corresponding probability for $\z$.) 

But perhaps a more interesting connection is that, as mentioned in \cite{gwrg}, the critical ($s=2$) case for  $\Gamma:= \z$ can be obtained as a special case of GWRG, just as the critical ($\alpha=b^2$) case for  $\Gamma:= \grpb$. 
This raises the question of whether there is a general method for finding critical generating measures for other groups, which we hope to explore in future work.	

\subsection{Outline of the paper}
This paper is structured as follows. In \Sr{secppmpem} we make precise the definition of our model via the canopy tree, and the closely-related model arising from GWRGs. In \Sr{proofthm} we prove \Prr{thminvariant} by showing that the cluster of the origin obtained by iterating \eqref{star} converges in distribution to that of the model defined in \Sr{secppmpem}. 

In \Sr{secperco} we analyse the model of \Sr{secintromu} more generally, showing that the exponent $\alpha=b^2$ is critical for percolation to occur (\Tr{critLRP}); we also give more precise bounds on the critical window (\Tr{zoomed thr}). In \Sr{size} we give the lower and upper bounds required for \Tr{Ebounds}.

In \Sr{secthresholds} we establish sharp thresholds for connectedness in the finite versions of the models of \Sr{secppmpem} (Theorems~\ref{diffthresh} and~\ref{logwindowpem}). 
In one version the threshold for connectedness coincides with the threshold at which no isolated vertices remain, but, perhaps surprisingly, this is not the case for the model arising from a GWRG. 

\section{Random graph models based on random walks on trees} \label{secppmpem}

In this section we provide the third alternative definition of the random graph $G(\la)$ from the introduction. This definition, called the `\pem', will be useful in our proof of \Prr{thminvariant} in \Sr{proofthm}. This \pem is closely related to an instance of  `group-walk random graphs' as introduced in \cite{gwrg}, which we define below as the `\ppm'. 

\subsection{Some notation for trees} \label{treenotation}

Fix an integer $b \geq 2$. We inductively define the $ b $-ary tree $T_h$ of height $h$ as follows. Let $T_0$ be the one-vertex tree with single vertex $v_0$. For each $h\geq 0$, $T_{h+1}$ is the graph obtained from $T_h$ by adding $ b -1$ additional copies of $T_h$ and a new vertex $v_{h+1}$, and adding edges between $v_{h+1}$ and the vertices of degree $ b $ (or $0$) in our $b$ copies of $T_h$. Note that $T_h$ has $ b ^h$ leaves and $( b ^{h+1}-1)/( b -1)$ vertices. Define the ($ b $-ary) \defi{canopy tree} $T_{\infty}$, as $T_{\infty}=\bigcup_hT_h$, where we think of $T_h$ as a subtree of $T_{h+1}$.
For $h\in \z[n] \cup \{\infty\}$, write $L_h$ for the set of leaves of $T_h$. 

Note that for each $h>0$ removing the vertex $v_h$ divides $T_{\infty}$ into $b$ finite components and one infinite component, which contains $v_{h+1}$. Therefore $T_{\infty}$ contains a unique infinite path starting at any vertex. Moreover, the group $\Gamma= \grpb$ from the introduction can be realised as a subgroup of the automorphism group of $T_{\infty}$ acting transitively and faithfully on the set of leaves $L_{\infty}$ of $T_{\infty}$.  Indeed, if we label the edges of $T_{\infty}$ with $0, 1, \ldots, b-1$, in such a way that for each non-leaf $v$, each label appears exactly once among the offspring of $v$, and all edges along the unique infinite path from $v_0$ are labelled $0$, then every element $g$ of $G$ can be identified with the unique leaf $v_g$ in $L_{\infty}$ such that the sequence of labels along the unique infinite path in $T_{\infty}$ starting at $v_g$ coincides with the sequence $g$. With this identification, multiplication with an element $h$ of $G$ defines an automorphism of $T_{\infty}$.

We define the \textit{height} $h(v)$ of a vertex $v$ in $T_{\infty}$ (or $T_h$) as the distance to the nearest leaf. The \textit{apex} of $T_h$ is the unique vertex of height $h$. Each vertex in $T_{\infty}$ has a unique higher neighbour, its \textit{parent}. We say that two distinct vertices are \textit{siblings} if they have the same parent. $x$ is a \textit{descendant} of $y$ if there is a path $x\cdots y$ in which each vertex except the first is the parent of the previous one; we include the possibility of a path of length $0$, so that $x$ is a descendant of itself. We say that $x$ is an \textit{ancestor} of $y$ if $y$ is a descendant of $x$. Write $D(v)$ for the descendants of $v$.

\subsection{The two models and their relationship} \label{secmodels}

We define a random multi-graph $G_n(\la)$ \fe\ $n\in\mathbb{N}\cup \{\infty\}$ and $\la \in \zp[r]$ as follows. The vertex set of $G_n(\la)$ is the set $L_n$  of leaves of $T_n$. For every pair $x,y\in L_n$, the number of $x$--$y$~edges in $G_n(\la)$ is a random variable with distribution $\Po{\la b ^{1-d(x,y)}}$ where $d(x,y)$ is the distance between $x,y$ in $T_n$. Note that in the case $n=\infty$ we have $\sum_{y\neq x}b^{1-d(x,y)}=1$.
Since $L_\infty$ can be identified with the group $\grpb$ (see the remark above), $G_\infty(\la)$ can be obtained as a special case of our \LRP\ model from \Sr{secintromu}.
We call $G_n(\la)$ the \defi{\pem}. If no $n$ is specified in the context, then the term \pem will refer to $G_\infty(\la)$.

\medskip
Next, we describe an instance of the random graph model of \cite{gwrg}, which will turn out to be very similar to the \pem.

Define a random multi-graph $\gnl$ 
\fe\ $n\in\mathbb{N}\cup \{\infty\}$ and $\la \in \zp[r]$ as follows. The vertex set of $\gnl$ is again the set $L_n$ of leaves of $T_n$. 
The edge set of $\gnl$ is determined by the following random experiment. 
At each vertex $v\in L_n$, we start a number of particles, and these numbers are i.i.d.\ random variables with distribution $\Po{\la/2}$. These particles perform simple random walks on $T_n$, and are stopped upon their first return to $L_n$. For each of these particles $p$, we put an edge $e_p$ in $\gnl$ connecting the vertex at which $p$ was started to the last vertex of its random walk. We call this random multi-graph $\gnl$ the \defi{{\ppm}}. (We choose $\la/2$ as the mean of the number of particles started at $v$ so that the number of edges at $v$, \ie the number of particles starting or ending at $v$, has mean $\la$.) 

\medskip
It is possible to show that $G_n(\la)$ converges in distribution to $G_\infty(\la)$ and likewise for $\gnl$ and $\gnl[\infty]$;\footnote{We thank Gourab Ray for this observation.} we will not use or prove this fact directly, but a lot of the intuition underlying this paper was based on it. 

It is not too hard to see that the expected number of $x$--$y$~edges in the \ppm decays like $\la b^{1-d(x,y)}$; in \Sr{calc} we provide some precise calculations. This implies that there are constants $0<c<C$ such that the \pem with parameter $\la$ lies between $\gnm[\infty]{c\la}$ and $\gnm[\infty]{C\la}$ \fe\ $\la$ (in the sense that the three processes may be coupled so that their edge sets are nested as $E(\gnm[\infty]{c\la}) \subseteq E(G_\infty(\la)) \subseteq E(\gnm[\infty]{C\la})$); see \Lr{zetabounds} and the remark after it.

\section{Proof of \Prr{thminvariant}} \label{proofthm}

For $n\in \z[n] \cup \{\infty\}$, let $C_n(\la)$ denote the component of $v_0$ -- the unique vertex of $T_0$ -- in the \pem $G_n(\la)$. Since the automorphism group of $T_n$ acts transitively on its leaves and this action preserves the probability distributions of edges between pairs of vertices, the component of any fixed vertex has the same distribution, up to isomorphism, as $C_n(\la)$.

First we show that $C_{\infty}(\la)$ is almost surely finite,\footnote{We thank Omer Angel for this observation.} by showing that there is almost surely some $n<\infty$ such that $V(C_{\infty}(\la))\subseteq L_n$, where as in \Sr{treenotation} $L_n$ denotes the set of leaves of $T_n$, which we think of as a subset of the leaves of $T_{n+1}$, and therefore of $T_\infty$. 
\begin{lem}\label{geogeo}For any $m>n$ (including the case $m=\infty$) we have $\prob{V(C_m(\la))\subseteq L_n}>1-(1-\ee^{-\la}/ b )^n$.\end{lem}
\begin{proof}It is sufficient to prove the case $m=\infty$, since we may obtain $G_m(\la)$ as an induced subgraph of $G_{\infty}(\la)$. 

For each $i\in\mathbb{N}$ let $e_i$ be the edge from $T_i$ to $T_{\infty}\setminus T_i$. For each $i$ write $E_i$ for the event that at least one edge of $G_{\infty}(\la)$ starts and finishes on different sides of $e_i$, \ie starts in $L_i$ and finishes outside it or vice versa. It suffices to show that
\[\probbb{\bigcap_{i=1}^nE_i}<(1-\ee^{-\la}/ b )^n\,,\]
since whenever $V(C_m(\la))\not\subseteq L_n$, the events $E_1,\ldots,E_n$ must all hold.

Note that $\prob{E_i}=(1-\ee[-\la])$. We claim that, given $\bigcup_{j\leq i}E_j$, the probability that every edge which crosses $e_i$ also crosses $e_{i+1}$ is less than $1/ b $. The number of edges crossing $e_i$ but not $e_{i+1}$ is given by a Poisson random variable of mean $\la(b-1)/b$, and the number of edges crossing both $e_i$ and $e_{i+1}$ by an independent Poisson random variable of mean $\la/b$. Thus, conditional on there being $k$ edges which cross $e_i$, the number of edges crossing both $e_i$ and $e_{i+1}$ has distribution $\operatorname{Bin}(k,1/b)$. The probability that every edge which crosses $e_i$ also crosses $e_{i+1}$ is therefore $b^{-k}$; conditional on $k\geq 1$ this is at most $1/b$.

Now we proceed as follows. Set $k_0=1$. For each $i$, we look sequentially, first for edges which cross $e_{k_i}$ but not $e_{k_i+1}$, then for edges which cross $e_{k_i},e_{k_i+1}$ but not $e_{k_i+2}$, and so on. If we eventually find such an edge, let $e_{k_{i+1}}$ be the first edge it doesn't cross and move on to $i+1$; otherwise stop the process and set $k=k_i$. Note that the event $E_{k_{i+1}}$ is independent of the information we have after finding an edge which crosses $e_{k_i}$, since edges between different pairs of vertices occur independently. Now $E_k$ is the first event which does not occur, and $k$ is dominated by the variable $X=\sum_{i=1}^YZ_i$, where $Y\sim\operatorname{Geo}(\ee[-\la])$ and $Z_i\sim\operatorname{Geo}(1/ b )$ are independent.\footnote{We adopt the convention that a geometric random variable represents the number of trials up to and including the first success.} Writing $p=\ee[-\la]$ and $q=1/ b $, the generating function of $Y$ is $f_Y(s)=ps/(1-(1-p)s)$, and that of each $Z_i$ is $f_Z(s)=qs/(1-(1-q)s)$. Then
\begin{align*}
f_X(s)&=f_Y(f_Z(s))\\
&=p\frac{qs}{1-(1-q)s}\cdot\frac{1}{1-(1-p)\frac{qs}{1-(1-q)s}}\\
&=\frac{pqs}{1-(1-q)s-(1-p)qs}\\
&=\frac{pqs}{1-(1-pq)s}\,,
\end{align*}
so $X\sim\operatorname{Geo}(pq)$, giving the required result.\end{proof}
\begin{rem}It follows that $C_n(\la)$ converges in distribution to $C_\infty(\la)$, since there is a natural coupling for which $C_n(\la)$ is always a subgraph of $C_\infty(\la)$, and for which \Lr{geogeo} gives almost sure convergence.\end{rem}

We now prove \Prr{thminvariant} by showing that $C_\infty(\la)$ is the unique random multi-graph with finite average degree invariant under \eqref{star}.

For this, note that if we add an extra layer $L_\infty'$ to $T_\infty$ by attaching $b$ new leaves $u_0,\ldots,u_{b-1}$ to each leaf $u$ of $T_\infty$, then the resulting tree $T_\infty'$ is isomorphic to $T_\infty$. Therefore, if we repeat the definition of $G_\infty(\la)$ using $T_\infty'$ instead of $T_\infty$, we obtain a random multi-graph $G_\infty'(\la)$ which is identically distributed with $G_\infty(\la)$. Moreover, it is straightforward to check that if we apply operation \eqref{star} to $G_\infty(\la)$ we obtain a realisation of $G_\infty'(\la)$, because of the choice of the rates $b^{1-2h}\la$ at which edges appear. It follows that if we let $G(\la)$ denote $C_\infty(\la)$, then $G(\la)$ is indeed invariant under \eqref{star}, as  \eqref{star} can be thought of as choosing the component of $o_0$ in $G_\infty'(\la)$, which is identically distributed with the component of $o$ in $G_\infty(\la)$.

To prove the uniqueness of $G(\la)$, let $(X,o)$ be another random rooted multi-graph -- possibly depending on $\la$ -- with these properties. Let $\delta$ denote the expected degree of $o$ in $X$, and recall that we are assuming that $\delta$ is finite. This means that with positive probability $q$, the root will become isolated if we perform operation \eqref{star} on $(X,o)$; indeed, the root becomes isolated whenever all edges of $o$ are inherited by $b-1$ of its offspring, no new edges are formed between these offspring and the remaining one $o_i$, and we choose $o_i$ as the new root. Moreover, if we perform  \eqref{star} once more on the resulting graph, then the probability to obtain an isolated root is still $q$. As the choices we make each time we apply \eqref{star} are independent from what happened in earlier applications, it follows that \fe\ $\eps>0$ there is $N\in \z[n]$ such that if we perform \eqref{star} $N$ times on $(X,o)$, then the probability to obtain an isolated root at least once is at least $1-\eps$. Note that if the root is an isolated vertex, then performing \eqref{star} $M$ times on it yields a random graph with the law of the component $C_M(\la)$. But as $C_M(\la)$ converges in distribution to $C_\infty(\la)$, and the distribution of $(X,o)$ is preserved after performing \eqref{star} any number of times, it follows that this distribution coincides with $C_\infty(\la)$. This proves \Prr{thminvariant}.
\medskip

We have just proved that $G(\la)$ coincides with $G_\infty(\la)$.  
Since the leaves $L_\infty$ of $T_\infty$ can be identified with the elements of the group $\grpb$ as remarked in \Sr{treenotation},  it follows that the \pem $G_\infty(\la)$ coincides with the percolation model on $\grpb$ defined in the introduction, and hence also with $G(\la)$ as claimed there. 
\medskip

We next show that the assumption that the root has finite expected degree is necessary for \Prr{thminvariant} to hold: without this restriction $G(\la)$ is not unique. 

\begin{proposition}There is a random connected rooted multigraph with infinite degrees which is invariant under \eqref{star}.\end{proposition}
\begin{proof}
Let $\Gamma^*$ be the group whose elements are the two-way infinite sequences $(x^{(i)})_{i\in\z}$ over $\z_b$ such that $\{i:x^{(i)}=1\}$ is finite and $x^{(-2i)}=0$ for every $i>0$; the group operation is componentwise addition in $\z_b$. Let $\mu^*(x)=b^{1-\ell'(x)}$, where $\ell'(x)=\min\{i:x^{(i)}=1\}$. Let $G$ be the random graph obtained by applying $\eqref{star}$ to $\Gamma^*_{\mu^*}(\la)$, where the vertices of $G$ are $\{x_0,\ldots,x_{b-1}\mid x\in\Gamma^*\}$; note that $G$ is almost surely connected. We may define a bijection between the vertices of $G$ and of $\Gamma^*_{\mu^*}(\la)$ as follows: $x_i\mapsto y$ where $y^{(0)}=i$ and $y^{(j)}=x^{j-2}$ for $j\neq 0$. It is easy to check that this bijection preserves the rates of all edges, and so $\Gamma^*_{\mu^*}(\la)$ is also invariant under \eqref{star}.\end{proof}

\section{Percolation on $\grpb$} \label{secperco}

The group $G:= \grpb$ consists of all sequences of elements of $\z_b$ which have only finitely many non-zero terms, endowed with the operation of componentwise addition, where $\z_b= \z / b\z$ is the cyclic group of order $b$ (the reader will lose nothing by assuming that $b=2$ throughout this section). 

In this section we study the question of whether percolation occurs in $G_\mu(\la)$ for large enough $\la$ for various generating measures $\mu$ on $G$. More precisely, the aim of this section is to determine the critical asymptotic decay for $\mu$ that separates the $\la_{\mathrm c}=\infty$ from the $\la_{\mathrm c} <\infty$ regime. We shall sometimes neglect to normalise $\mu$ to be a probability measure; this does not affect the results of this section since normalising $\mu$ is equivalent to rescaling $\la$. We use standard asymptotic notation: if $f,g:[0,\infty)\to[0,\infty)$ we write 
\[f(x)=\begin{cases}o(g(x))&\text{ if }\lim_{x\to\infty}f(x)/g(x)=0;\\
O(g(x))&\text{ if }\limsup_{x\to\infty}f(x)/g(x)<\infty;\\
\Omega(g(x))&\text{ if }\liminf_{x\to\infty}f(x)/g(x)>0;\\
\Theta(g(x))&\text{ if }f(x)=O(g(x))\text{ and }f(x)=\Omega(g(x)).\end{cases}\]

For $y \in G$, write $\ell(y)$ for the position of the last non-zero term of $y$ (we take $\ell(\boldsymbol{0})=0$). We will concentrate on generating measures $\mu(y)$ that depend on $\ell(y)$ only, and are monotone decreasing in $\ell(y)$: given a real number $\alpha> b$, we define a generating measure $\mu_\alpha$ by letting $\mu(y) = \mu_\alpha(y):= \alpha^{- \ell(y)}$. 
The reason why we do not consider $\alpha\leq b$ is that $\mu_\alpha$ fails to be a finite measure in that case, since $G$ has $b^{k-1}$ elements $y$ with $\ell(y)=k$.

We consider the percolation process $G_{\mu_\alpha}(\la)$ as defined in \Sr{secintromu}. 
In this case percolation does not occur for too small $\la$, since an exploration of the component of the identity is dominated by a subcritical Galton--Watson tree.

The Poisson edge model of \Sr{secmodels} is obtained by taking $\alpha=b^2$, and by \Lr{geogeo} there is no percolation, \ie $\la_{\mathrm c}=\infty$, in this case. Easily, the same arguments imply that $\la_{\mathrm c}=\infty$ when $\alpha>b^2$ as well. Our next result shows that this value $\alpha=b^2$ is in a sense critical.

\begin{thm}\label{critLRP}
For $\alpha\in(b,\infty)$, the percolation model $G_{\mu_\alpha}(\la)$ satisfies $\la_{\mathrm c}<\infty$ if $\alpha<b^2$ and $\la_{\mathrm c}=\infty$ if $\alpha\geq b^2$. 
\end{thm}
\begin{proof}The second statement follows from \Lr{geogeo}, and so we may assume $\alpha<b^2$.
Write $V_k$ for $\{x\in G \mid \ell(x)\leq k\}$, and $X_k$ for the component of the identity in the subgraph of $G_{\mu_\alpha}(\la)$ spanned by $V_k$. Fix a real $\beta$ with $\alpha/b<\beta<b$. We aim to show that for large enough $\la$ we have $\prob{\abs{X_k}\geq\beta^k \text{ \fe\ } k}>0$. In fact, we shall show that $\prob{\abs{X_{k_i}}\geq k_i\beta^{k_i}\text{ for every } i}>p$ for some $p>0$ and a strictly increasing sequence $(k_i)_{i\geq 0}$ chosen so that 
\begin{equation}\label{ki}\beta^{k_{i+1}}\leq k_i\beta^{k_i}\,.\end{equation}
This is sufficient since if $k_i<k<k_{i+1}$ then $\abs{X_k}\geq\abs{X_{k_i}}\geq k_i\beta^{k_i}\geq\beta^{k_{i+1}}>\beta^k$. We defer the choice of $k_0$ until later, but given $k_0$, we will choose $(k_i)_{i\geq 1}$ to be as large as possible given \eqref{ki}; this will mean that $k_{i+1}=\floor{k_i+\log_{\beta}k_i}$. 

Write $p_i$ for $\prob{\abs{X_{k_j}}\geq k_j\beta^{k_j}\text{ for every } j\leq i}$. Now $V_{k_{i+1}}$ consists of $b^{k_{i+1}-k_i}$ cosets of $V_{k_i}$, and $b^{k_{i+1}-k_i}\approx b^{\log_{\beta}k_i}=k_i^{\log_{\beta}b}$; note that $\log_{\beta}b>1$.

The number of edges from $X_{k_i}$ to any given coset $xV_{k_i}$, where $k_i<\ell(x)\leq k_{i+1}$, is given by a Poisson random variable with mean at least $\la \abs{X_{k_i}}b^{k_i}\alpha^{-k_{i+1}}$. So if $\abs{X_{k_i}}\geq k_i\beta^{k_i}\geq\beta^{k_{i+1}}$, then the probability that there is no such edge is at most $q_i:=\exp\bigl(-\la\bfrac{b\beta}{\alpha}^{k_{i+1}}b^{k_i-k_{i+1}}\bigr)$. 

Call a coset $xV_{k_i}$ \textit{good} if there is an edge from $X_{k_i}$ to a component of size at least $k_i\beta^{k_i}$ in the graph restricted to $xV_{k_i}$. Thus each coset independently has probability at least $(1-q_i)p_i$ of being good, and $(1-q_i)p_ik_i^{\log_{\beta}b}>2k_{i+1}$ provided $p_i>p$ and $k_0$ is sufficiently large. Consequently, by the multiplicative Chernoff bound the probability of at least $k_{i+1}$ cosets being good is at least $1-c^{k_{i+1}}$ for some absolute constant $c<1$. This event is sufficient to imply that $\abs{X_{k_{i+1}}}\geq k_{i+1}\beta^{k_{i+1}}$, so
\[p_{i+1}\geq p_i(1-c^{k_{i+1}})\,\]
for each $i\geq 0$, and hence
\[p_{i}\geq p_0\prod_{j=1}^i(1-c^{k_j})\geq p_0\prod_{k=k_0}^{\infty}(1-c^{k})\,.\]
We may choose $k_0$ sufficiently large that $\prod_{k=k_0}^{\infty}(1-c^{k})>1-p$, and then if $p_0$ is sufficiently close to $1$ this will imply $p_i>p$ for all $i$. But by choosing $\la$ appropriately we may ensure $p_0$ is large enough.\end{proof}

Thus we have proved that for generating measures of the form $\mu(y) = \alpha^{- \ell(y)}$, the value $\alpha = b^2$ is critical for the occurrence of percolation for 
large enough $\la$. Our next result `zooms into' the critical case $\alpha = b^2$ by considering measures $\mu(y)$ that decay as $(b^2-o(1))^{-\ell(y)}$. We give a class 
of such measures for which percolation does occur. Our previous proof of non-percolation in the critical case was based on the existence of arbitrarily large sets of constant 
`boundary', i.e.\ finite subsets $S_k$ of $G$, namely those of the form $S_k:=\{y \mid \ell(y)\leq k\}$, such that the total $\mu$-measure of edges with exactly one 
endvertex in $S_k$ is bounded above. We will show that in some cases there is no percolation even though the minimum size of the boundary of a set of size $n$ tends 
to infinity with $n$.

\begin{thm} \label{zoomed thr}
Let $\mu(y) = b^{-2\ell(y)}f(\ell(y))$, where $f(\ell)$ is an increasing function. Then the percolation model $G_{\mu}(\la)$ satisfies
\begin{enumerate}[(a)]
\item $\la_{\mathrm c}<\infty$ if $f(\ell)=\Omega\bigl(a^{\sqrt{\ell}}\bigr)$ for some $a>0$;
\item $\la_{\mathrm c}=\infty$ if $f(\ell)=o(\log\ell)$ and $f(\ell)$ is ultimately concave.
\end{enumerate}
\end{thm}
\begin{rem}In particular, if $f(\ell)=\frac{\log\ell}{\log\log\ell}$ then percolation does not occur even though large sets of constant boundary do not exist.\end{rem}
\begin{proof}For (a), we follow a similar argument to \Tr{critLRP}. Define a sequence $(k_i)_{i\geq 0}$ where $k_0$ is to be chosen later, and for each $i\geq 0$ we choose $k_{i+1}$ as large as possible so that $k_{i+1}-k_i\leq (\log_ba)\sqrt{k_{i+1}}/4$ (we will choose $k_0$ sufficiently large that $k_{i+1}>k_i$ for each $i$). We aim to show that $\prob{\abs{X_{k_i}}\geq b^{k_i}a^{-\sqrt{k_i}/4}\text{ for every } i}>p$ for some fixed $p>0$. Provided $X_{k_i}\geq b^{k_i}a^{-\sqrt{k_i}/4}$, we again consider the cosets $xV_{k_i}$ contained in $V_{k_{i+1}}$. Call $xV_{k_i}$ \textit{good} if there is an edge from $X_{k_i}$ to a component of size at least $b^{k_i}a^{-\sqrt{k_i}/4}$ in the graph restricted to $xV_{k_i}$. Provided there are at least $a^{\sqrt{k_i}/4}$ good cosets, we must have $\abs{X_{k_{i+1}}}\geq b^{k_i}\geq b^{k_{i+1}}a^{-\sqrt{k_{i+1}}/4}$.

Again, provided $X_{k_i}\geq b^{k_i}a^{-\sqrt{k_i}/4}$, the number of edges from $X_{k_i}$ to any given coset of $V_{k_i}$ within $V_{k_{i+1}}$ is given by a 
Poisson random variable with mean at least 
\[\la b^{k_i}a^{-\sqrt{k_i}/4}b^{k_i}a^{\sqrt{k_{i+1}}}b^{-2k_{i+1}}>\la a^{3asqrt{k_{i+1}}/4}\,;\]
again we write $q_i=\exp(-\la a^{3\sqrt{k_{i+1}}/4})$ and $p_i=\prob{\abs{X_{k_j}}\geq b^{k_j}a^{-\sqrt{k_j}/4}\text{ for every } j<i}$. Then the expected number of good cosets is at least $(1-q_i)p_i b^{k_{i+1}-k_i}$. 

Since $\ceil{(\sqrt{k_i}+\log_ba/10)^2}<k_i+a\sqrt{k_i}/4$ provided $k_0$ (and hence $k_i$) is sufficiently large, by a suitable choice of $k_0$ we may ensure that $\sqrt{k_{i+1}}>\sqrt{k_i}+\log_ba/10$ for all $i$. We may also choose $k_0$ and $p$ large enough that $(1-q_i)pa^{1/10}>1+\eps$ for all $i$, for some $\eps>0$. Then by the Chernoff bound the probability of at least $a^{\sqrt{k_i}/4}$ good cosets exceeds $1-\exp(-\delta b^{k_{i+1}-k_i})$ for some $\delta>0$. As before, we can find some sufficiently large $k_0$ and some $p'$ sufficiently close to $1$ that $p_0>p'$ will ensure $p_i>p$ for each $i$. Since $p_0>p'$ for some sufficiently large $\la$, this completes the proof of (a).

For (b), we use a similar approach to \Lr{geogeo}. In order for the component of the identity to be infinite, there must be an edge crossing $e_k$ for every $k$; write $E_k$ for the event that $e_k$ is crossed. The total rate at which edges crossing $e_k$ appear is \[(b-1)\sum_{h\geq k} b^{k-h}f(h)=b\sum_{j\geq 1}(1-1/b)(1/b)^{j-1}f(k+j-1)=b\mean{f(k+J-1)}\,,\]
where $J$ is a geometric random variable with mean $b/(b-1)$. Since $f$ is ultimately concave, this is at most $bf(k+1)$, and so $\prob{E_k}\leq1-\ee[-\la bf(k+1)]$, for every sufficiently large $k$.

Define a random sequence $k_i$ where $k_0$ is a constant to be chosen later. For each $i$, we look sequentially, first for edges which cross $e_{k_i}$ but not $e_{k_i+1}$, then for edges which cross $e_{k_i},e_{k_i+1}$ but not $e_{k_i+2}$, and so on. 
If we eventually find such an edge, let $e_{k_{i+1}}$ be the first edge it doesn't cross and move on to $i+1$; otherwise stop the process. The event $E_{k_{i+1}}$ only depends on the presence or absence of edges which have not yet been revealed, so it is independent of the previous process. Also, the rate at which edges crossing $\seq[k_i]e{k_i+h}$ but not $e_{k_i+h+1}$ appear is $(b-1)b^{-h}f(k_i+h)$. Since $f$ is ultimately concave and increasing, $f(k+1)=(1+o(1))f(k)$, so provided $k_0$ is sufficiently large, edges crossing $\seq [k_i]e{k_i+h}$ but not $e_{k_i+h+1}$ appear at least $3/2$ times as often as edges that cross $\seq[k_i]e{k_i+h+1}$ but not $e_{k_i+h+2}$ (any constant between $1$ and $b$ would do). 
Thus $\prob{k_{i+1}=k_i+h+1}\geq 3\prob{k_{i+1}=k_i+h+2}/2$, and, conditional on the existence of $k_{i+1}$, we can bound the distribution of $k_{i+1}-k_i$ by a geometric random variable with fixed mean $m$. Thus the probability that $k_{i}$ exists and exceeds $k_0+2im$ is bounded by the probability that a negative binomial random variable exceeds twice its mean, which is at most $c^i$ for some constant $c<1$ by the Chernoff bound. 
Consequently, \[\probbb{\bigcup_{j\leq i}E_{k_j}\wedge (k_i>k_0+2im)}\leq c^i\]
and \[\probbb{\bigcup_{j\leq i}E_{k_j}\wedge (k_i\leq k_0+2im)}\leq\probbb{\bigcup_{j<i}E_{k_j}}(1-\exp(-\la bf(k_0+2im+1))\,,\]
so, writing $P_i=\probb{\bigcup_{j\leq i}E_{k_j}}$, we have $P_i\leq(1-\exp(-\la bf(k_0+2im+1)))P_{i-1}+c^i$. Since $f(k)=o(\log k)$, and $b$ and $\la$ are fixed, if $i$ is sufficiently large $\exp(-\la bf(k_0+2im+1))>(k_0+2im+1)^{-1}$. Thus $\sum_{i\geq 1}\exp(-\la bf(k_0+2im+1))=\infty$, and so \Lr{sequences} below gives the required result.
\end{proof}

\begin{lem}\label{sequences}
Let $(x_i)_{i\geq 0}$ and $(y_i)_{i\geq 1}$ be sequences of positive real numbers with $y_i<1$ for each $i$. Define $(z_i)_{i\geq 0}$ as follows: $z_0=x_0$ and $z_{i}=x_i+(1-y_i)z_{i-1}$ for $i>0$. Then provided $\sum_ix_i<\infty$ and $\sum_iy_i=\infty$, we have $\lim_{i\to\infty}z_i=0$.
\end{lem}
\begin{proof}
Fix $\eps>0$. For some $N$, $\sum_{i\geq N}x_i<\eps/2$. For each $i<N$, choose $n_i$ sufficiently large that $x_i\prod_{j=i+1}^{n_i}(1-y_j)<\eps/2N$; this is possible since $\prod_{j>i}(1-y_j)=0$. Now if $n\geq\max\{N,\max_i\{n_i\}\}$, then 
\begin{align*}z_n&=\sum_{i\leq n}\Bigl(x_i\prod_{j=i+1}^n(1-y_j)\Bigr)\\
&\leq \sum_{i<N}\Bigl(x_i\prod_{j=i+1}^{n_j}(1-y_j)\Bigr)+\sum_{i\geq N}x_i\\
&<\eps\,.\qedhere\end{align*}
\end{proof}

\section{Expected size of $G(\la)$} \label{size}

The aim of this section is to prove \Tr{Ebounds}. Here we will prove the corresponding bounds for $C_\infty(\la)$, the component of $v_0$ in the \pem $G_\infty(\la)$ defined in \Sr{secmodels}, which we proved in \Sr{proofthm} coincides with $G(\la)$. 

\subsection{Lower bound}
The lower bound of \Tr{Ebounds} is given by the following result.
\begin{proposition} \label{thmLB}
There exists $k>0$ such that $\mean{\abs{C_\infty(\la)}}=\Omega\bigl(\ee[k\la]\bigr)$ as $\la\to\infty$.
\end{proposition}
In order to prove \Prr{thmLB} we will need the following result on coupled variables with identical distributions. 
\begin{lem}\label{powerb}Suppose $X$ and $Y$ are identically distributed variables taking values in $\mathbb{Z}^+$, coupled such that $\prob{X=b  Y}\geq 1-p$. Then $\mean{X}\geq f(p)$ where  $f(1/n)=\frac{ b ^n-1}{n( b -1)}$ for $n\in\mathbb{Z}^+$ and $f(p)$ is linear between such points (so $f(p)\geq p\frac{ b ^{1/p}-1}{ b -1}$).\end{lem}
For ease of reading we proceed directly with the proof of \Prr{thmLB}, proving \Lr{powerb} afterwards.
\begin{proof}[Proof of \Prr{thmLB}]
We will prove that
\[\mean{\abs{C_\infty(\la)}}\geq\frac{(\ee[( b -1)\la/b]-1)b\log b }{( b -1)^2\la}\,,\]
for $\la$ sufficiently large.

As in \Sr{proofthm}, applying operation \eqref{star} to $G_{\infty}(\la)$ (without discarding other components) gives an identically-distributed random graph $G'_{\infty}(\la)$, and the components of the roots $C_{\infty}(\la)$ and $C'_{\infty}(\la)$ are likewise identically distributed.

For $v\in B_1$, let $F(v)$ be the event that the subgraph of $G'_\infty(\la)$ induced by the children of $v$ is disconnected. For $b=2$ this is just the event that there are no edges between the two children of $v$, so has probability $\ee[-\la/b]$. If $b>2$ we may bound $\prob{F(v)}$ by the probability of the event that some child of $v$ has no edges to the other children of $v$. By Bonferroni's inequality this occurs with probability at least $b\ee[-(b-1)\la/b]-\binom b2\ee[-(2b-3)\la/b]=(b-o(1))b\ee[-(b-1)\la/b]$ as $\la\to\infty$ (since $2b-3>b-1$). Thus for any $b\geq 2$ and $\la$ sufficiently large we have $\prob{F(v)}\geq( b -1)\ee[-( b -1)\la/b]$.

We consider two cases.
Suppose first that $\prob{\exists v\in C_{\infty}(\la):F(v)\text{ occurs}}\geq\frac{b\log b }{( b -1)\la}$. Then, for $\la$ sufficiently large
\begin{align*}
\prob{\exists v\in C_{\infty}(\la):F(v)\text{ occurs}}&\leq\mean{\#\{v\in C_{\infty}(\la):F(v) \text{ occurs}\}}\\
&=\mean{\abs{C_{\infty}(\la)}}\prob{F(v)}\\
&=\mean{\abs{C_{\infty}(\la)}}( b -1)\ee[-( b -1)\la/b]\,,
\end{align*}
since the events $(F(v))_{v\in B_1}$ are independent of $C_{\infty}(\la)$. So we must have $\mean{\abs{C_{\infty}(\la)}}\geq\frac{\ee[( b -1)\la/b]b\log b }{( b -1)^2\la}$.

The second case is where $\prob{\exists v\in C_{\infty}(\la):F(v)\text{ occurs}}\leq\frac{b\log b }{( b -1)\la}$. Then the complementary event, that $F(v)$ fails for all $v\in C_{\infty}(\la)$, occurs with probability greater than $1-\frac{b\log b }{( b -1)\la}$, and in this event we have $\abs{C'_\infty(\la)}= b \abs{C_{\infty}(\la)}$. But these two variables are identically distributed, and so, by \Lr{powerb}, 
\[\mean{\abs{C_{\infty}(\la)}}\geq\bfrac{ b ^{\frac{( b -1)\la}{b\log b }}-1}{ b -1}\bfrac{b\log b }{( b -1)\la}=\frac{(\ee[( b -1)\la/b]-1)b\log b }{( b -1)^2\la}\,,\]
as required.
\end{proof}
\begin{rem}The same argument can be used to prove that if $C^-_\infty(\la)$ is the component of $v_0$ in the subgraph of $G_{\infty}(\la)$ obtained by deleting all edges between siblings of $T_\infty$, then $\mean{C^-_\infty(\la)}$ is also at least exponential in $\la$.
\end{rem}
We now give the deferred proof of \Lr{powerb}
\begin{proof}[Proof of \Lr{powerb}]For each $m$ which is not divisible by $ b $, write $Z_m$ for the set $\{m, b  m, b ^2m,\ldots\}$. Suppose that $\prob{X\in Z_m}>0$ and $\cprob{X= b  Y}{X\in Z_m}=1-q_m$.

First we claim $\max_{x\in Z_m}\cprob{X=x}{X\in Z_m}\leq q_m$. Suppose not, and $\cprob{X=x}{X\in Z_m}=q_m+\eps$. For some $y>0$ we have $\cprob{X= b ^yx}{X\in Z_m}<\eps$. Now
\begin{align*}
\prob{(X\neq b  Y)\wedge(X\in Z_m)}&\geq\sum_{z=0}^{y-1}\prob{X= b ^zx\wedge Y\neq b ^{z+1}x}\\
&\geq\sum_{z=0}^{y-1}(\prob{X=b^zx}-\prob{Y= b ^{z+1}x})\\
&=\prob{X=x}-\prob{X= b ^yx}\\
&>q_m\prob{X\in Z_m}\,.\end{align*}
But $\prob{(X\neq b  Y)\wedge(X\in Z_m)}=q_m\prob{X\in Z_m}$. This proves the claim. Now it follows that $\cprob{X> b ^km}{X\in Z_m}\geq 1-kq_m$, and so the variable $X\mid X\in Z_m$ dominates the variable which takes the value $ b ^km$ with probability $q_m$ if $k\leq\floor{q_m^{-1}}$, $0$ if $k>\floor{q_m^{-1}}+1$, and $1-q_m\floor{q_m^{-1}}$ if $k=\floor{q_m^{-1}}$. This variable has expectation $mf(q_m)$. 

Finally, $p=\sum_m\prob{X\in Z_m}q_m$, and $\mean{X}=\sum_m\prob{X\in Z_m}mf(q_m)\geq\sum_m\prob{X\in Z_m}f(q_m)$. This is at least $f(p)$ by Jensen's inequality.\end{proof}

\subsection{Upper bound}
In order to bound the expected size of $C_{\infty}(\la)$, we first show that this is equivalent to bounding the limit as $n\to\infty$ of the expected size of $C_n(\la)$. 

\begin{lem}\label{limcomp}For any $m>n$ (including the case $m=\infty$) we have
\[\mean{\abs{C_n(\la)}}<\mean{\abs{C_m(\la)}}<(1+o(1))\mean{\abs{C_n(\la)}}\]
as $n\to\infty$.\end{lem}
\begin{rem}In particular, taking $m=\infty$ and $n$ sufficiently large gives $\mean{\abs{C_{\infty}(\la)}}<\infty$.\end{rem}
\begin{proof}
The first inequality follows from the fact that $G_n(\la)$ can be obtained as an induced subgraph of $G_m(\la)$. For the second inequality, we can obtain $G_m(\la)$ by first running $G_n(\la)$ on each copy of $T_n$, then adding edges between copies. We know from \Lr{geogeo} that the probability that the component (in $G_n(\la)$) of a randomly-chosen vertex in a copy of $T_n$ having an edge in the second stage is $o(1)$. If it does, the expected number of additional edges it has in the second stage is at most the expected number of edges from the whole of that copy of $T_n$, which is about $\la$. So the expected number of edges the component of a randomly-chosen vertex receives in the second stage is also $o(1)$. We now consider a branching process: start from the component of $v_0$ in its copy of $T_n$; the next generation is the set of components we reach with one long edge, and so on. We can bound this from above by assuming that all edges go to new, different components. This is subcritical, since the expected number of offspring is $o(1)$, and so the total expected number of components merged is $1+o(1)$. Each component has expected size at most $\mean{\abs{C_n(\la)}}$, giving the required bound.\end{proof}

The upper bound of \Tr{Ebounds} follows from the following result.
\begin{proposition}$\log\mean{\abs{C_\infty(\la)}}=O(\ee[\la]\log\la)$.\end{proposition}
\begin{proof}By \Lr{limcomp}, it is sufficient to show that this bound applies to $\mean{\abs{C_h(\la)}}$ for every $h<\infty$. 
Write $d_h(\la)=\mean{\abs{C_h(\la)}}/\abs{L_h}$. Then $d_h(\la)\to 0$ as $h\to\infty$, since by \Lr{geogeo} there exists $k$ such that 
the probability that the component escapes $L_k$ is at most $\eps$, and so for $h$ sufficiently large we have $d_h(\la)<\eps+ b ^{k-h}<2\eps$. 

In particular, let $k=\lceil b \log((2b+4)\la)\ee[\la]\rceil$ and $h=k+\log_ b ((2b+4)\la)$. \Lr{geogeo} gives
\begin{align*}\prob{C_h\not\subseteq L_k}&<(1-\ee[-\la]/ b )^{ b \log((2b+4)\la)\ee[\la]}\\
&<(\ee[-1])^{\log((2b+4)\la)}\\
&=\la^{-1}/(2b+4)\,.\end{align*}
Therefore $d_{h}(\la)<\la^{-1}/(2b+4)+ b ^{k-h}=\la^{-1}/(b+2)$.

Suppose $\la d_h(\la)<( b ^{j}+ b^2 /( b+1))^{-1}$, for some $j\geq 0$. We can bound $d_{h+1}(\la)$ by first generating two independent copies of $G_h(\Po{\la})$, and then adding edges between the two, with each pair getting $\Po{b^{1-2h}\la}$ edges (call these the ``long'' edges). Now we use a similar branching-process argument to \Lr{limcomp}: the component of $v_0$ is $C_{h+1}(\la)=\bigcup_{i\geq 1}C_i$, where $C_1$ is the component of the root in the first half, $C_2$ is the union of all components in the second half having a long edge to $C_1$, $C_3$ is the union of components in the first half which are not in $C_1$ but are joined by a long edge to $C_2$, and so on. The expected number of long edges between $C_i$ and $C_{i+1}$, and therefore also the expected number of components in $C_{i+1}$, is at most $\la b ^{1-h}\abs{C_i}$. The expected total size of $C_{i+1}$ is therefore at most $b\la d_h(\la)\abs{C_i}$, since every time we reach a new vertex, the expected size of the component containing that vertex is at most $ b ^hd_h(\la)$, as all edges are independent. So $\mean{\la\abs{C_i} b ^{-h}}\leq\frac 1b(b\la d_h(\la))^i$, and 
\begin{align*}
\la d_{h+1}(\la)&=\frac{1}{ b }\sum_{i\geq 1}\mean{\la b ^{-h}\abs{C_i}}\\
&\leq\frac{1}{b^2}\sum_{i\geq 1}(b\la d_h(\la))^i\\
&=\frac{\la d_h(\la)}{ b (1-b\la d_h(\la))}\\
&<\frac{1}{ b ^{j}+ b^2 /( b+1)}\cdot\frac{ b ^{j}+ b^2 /( b -1)}{ b ( b ^{j}+b/( b -1))}\\
&=\frac{1}{ b ^{j+1}+ b^2 /( b -1)}\,.
\end{align*}
Consequently $\la d_{h+i}(\la)<( b ^{j+i}+ b^2 /( b -1))^{-1}$ and so 
\begin{align*}\mean{C_\infty(\la)}&\leq\lim_{i\to\infty}  b ^{h+i}d_{h+i}(\la)\\
&<\lim\la^{-1} b ^{h+i}/( b ^{j+i}+ b^2 /( b -1))\\
&=\la^{-1} b ^{h-j}\,.\end{align*}
In particular we have $j=0$ for $h=\lceil b \log((2b+4)\la)\ee[\la]\rceil+\log_ b ((2b+4)\la)$, and so $\log\mean{\abs{C_\infty(\la)}}\leq\log(\la^{-1} b ^h)=O(\ee[\la]\log\la)$.
\end{proof}
This completes the proof of \Tr{Ebounds}.

\section{Calculations on the \ppm} \label{calc}
In this section we provide some calculations concerning the behaviour of the random walks used in the definition of the \ppm, which we will use in \Sr{secthresholds} where we study connectedness thresholds for the finite versions of the model. These calculations will also show that $\gnl[\infty]$ may be coupled between $G_{\infty}(c\la)$ and $G_{\infty}(C\la)$ for some suitable $C>c>0$, and so the results of \Sr{size} also apply to the \ppm.

\medskip 
Recall that edges of the \ppm are given by random walks on $T_h$, for some fixed $h\in\z[n]\cup\{\infty\}$, starting and ending in $L_h\subseteq L_{\infty}$.
For $x,y\in L_{\infty}$, write $x\wto y$ for the event that a random walk starting from $x$ finishes at $y$. When discussing the probability of such events we use $\mathbb{P}_h$ to indicate that the random walk in question occurs in $T_h$. Note that the probability that a particle, starting from a vertex of height $1$, reaches height $k\leq h$ before height $0$ is $( b -1)/( b ^k-1)$. This is because, provided the walk has not yet reached height $h$, $ b ^{h(u_i)}$ is a martingale, where $u_i$ is the position of the particle after $i$ steps, and so if we run the process until the first time $\tau$ with $h(u_\tau)\in\{0,k\}$, we have 
\[ b ^k\prob{h(v_\tau)=k}+\prob{h(u_\tau)=0}= b \,.\]

\begin{lem}\label{edgeprob}Let $x$ and $y$ be two vertices with lowest common ancestor having height $k$ (if $x=y$ we have $k=1$). Then \[\prob[\infty]{x\wto y}=\sum_{h\geq k}\frac{( b -1)^2}{( b ^{h+1}-1)( b ^h-1)}\,.\]
\end{lem}
\begin{proof}We condition on the maximum height $H$ reached. As before, $\prob{H\geq h}=( b -1)/( b ^h-1)$, so 
\[
\prob{H=h}=\frac{ b -1}{ b ^h-1}-\frac{ b -1}{ b ^{h+1}-1}=\frac{( b -1)^2 b ^h}{( b ^h-1)( b ^{h+1}-1)}\,.
\]
Now given that $H=h$, any of the $ b ^h$ vertices having a common ancestor with $x$ at height $h$ are equally probable endpoints of the walk. So
\begin{align*}
\prob{x\wto y}&=\sum_{h\geq k}\prob{H=h}\cprob{x\wto y}{H=h}\\
&=\sum_{h\geq k}\frac{( b -1)^2}{( b ^{h+1}-1)( b ^h-1)}\,.\qedhere
\end{align*}\end{proof}
Write $\zeta_h$ for $\prob[\infty]{x\wto y}$ where $x$ and $y$ have lowest common ancestor at height $h$. We can use \Lr{edgeprob} to get good bounds on $\zeta_h$.
\begin{lem}\label{zeta}
For every $1\leq i< h$, we have $\zeta_{i+1}<\zeta_i/ b ^2$.
\end{lem}
\begin{proof}
Note that
\begin{align*}
\zeta_i&=\sum_{j\geq i}\frac{( b -1)^2}{( b ^j-1)( b ^{j+1}-1)}\\
&<\sum_{j\geq i}\frac{( b -1)^2}{( b ^j- b ^{j-i})( b ^{j+1}- b ^{j-i})}\\
&=\sum_{j\geq i}\frac{( b -1)^2}{ b ^{2j-2i}( b ^i-1)( b ^{i+1}-1)}\\
&=\frac{( b -1)^2}{( b ^i-1)( b ^{i+1}-1)}\sum_{j\geq 0} b ^{-2j}\\
&=\frac{( b -1)^2 b ^2}{( b ^i-1)( b ^{i+1}-1)( b ^2-1)}\,.
\end{align*}
Consequently
\[
\zeta_i-\zeta_{i+1}=\frac{( b -1)^2}{( b ^i-1)( b ^{i+1}-1)}>\frac{ b ^2-1}{ b ^2}\zeta_i\,,
\]
so $\zeta_{i+1}<\zeta_i/ b ^2$.
\end{proof}

\begin{rem}
If leaves $\seq x{ b }$ are siblings, then $\zeta_1=\prob{x_1\wto x_j}$ for each $i$, so we must have $\zeta_1<1/ b $ whence $\zeta_i< b ^{1-2i}$ for each $i$.
\end{rem}

Our next lemma provides upper and lower bounds on $\zeta_h$, showing that it decays roughly like $b ^{1-2h}$.
\begin{lem}\label{zetabounds}
For each $h\geq 1$ we have
\[\bfrac{ b -1}{ b +1}( b ^{1-2h}+ b ^{1-3h})<\zeta_h<\Bigl(1+\frac{1}{ b ^{h}-1}\Bigr)\Bigl(1+\frac{1}{ b ^{h+1}-1}\Bigr)\frac{ b -1}{ b +1} b ^{1-2h}\,.\]\end{lem}
\begin{proof}
By \Lr{edgeprob}, we have
\[\zeta_h=\sum_{j\geq h}\Bigl(1+\frac{1}{ b ^j-1}\Bigr)\Bigl(1+\frac{1}{ b ^{j+1}-1}\Bigr)\frac{( b -1)^2}{ b ^{2j+1}}\,.\]
The first two factors are close to $1$, and we bound them by their maximum values to get
\begin{align*}
\zeta_h&<\sum_{j\geq h}\Bigl(1+\frac{1}{ b ^h-1}\Bigr)\Bigl(1+\frac{1}{ b ^{h+1}-1}\Bigr)\frac{( b -1)^2}{ b ^{2j+1}}\\
&=\Bigl(1+\frac{1}{ b ^h-1}\Bigr)\Bigl(1+\frac{1}{ b ^{h+1}-1}\Bigr)\frac{ b -1}{ b +1} b ^{1-2h}\,.
\end{align*}
Moreover, noting that $\frac{1}{m-1}=\frac 1m+\frac 1{m(m-1)}$,
\begin{align*}
\zeta_h&=\sum_{j\geq h}\frac{( b -1)^2}{( b ^j-1)( b ^{j+1}-1)}\\
&=( b -1)^2\sum_{j\geq h}\biggl(\frac{1}{ b ^j}+\frac{1}{ b ^j( b ^j-1)}\biggr)\biggl(\frac{1}{ b ^{j+1}}+\frac{1}{ b ^{j+1}( b ^{j+1}-1)}\biggr)\,.
\end{align*}
We may bound $\frac{1}{b^j(b^j-1)}>\frac{1}{b^{2j}}$ and expand (neglecting the lowest-order term) to get
\begin{align*}
\zeta_h&>( b -1)^2\sum_{j\geq h}\bigl( b ^{-2j-1}+ b ^{-3j-1}+ b ^{-3j-2}\bigr)\\
&=\frac{ b -1}{ b +1} b ^{1-2h}+ b ^{h-1}\frac{ b -1}{ b ^2+ b +1}\bigl( b ^{2-3h}+ b ^{1-3h}\bigr)\\
&>\bfrac{ b -1}{ b +1}( b ^{1-2h}+ b ^{1-3h})\,.\qedhere
\end{align*}
\end{proof}

The independence between the particles implies that the number of particles going from $x$ to $y$ has a Poisson distribution with rate $\frac{\la}{2}\prob{x\wto y}$, and so the total number of edges between $x$ and $y$ is a Poisson random variable with mean $\la\prob{x\wto y}$. Furthermore, the total number of edges $vw$ and the total number of edges $xy$ are independent Poisson random variables whenever $\{v,w\}\neq\{x,y\}$. 
\begin{rem}
Since both the lower and upper bounds for $\zeta_h$ are $\Theta(b^{1-2h})$, this proves our claim that by choosing the parameters $\la$ appropriately, 
the \pem can dominate the \ppm and vice-versa.
\end{rem}

\medskip
In $\gnl[\infty]$, the number of edges between $x$ and $y$ has distribution $\Po{\zeta_k\la}$ where $k$ is the height of the lowest common ancestor of $x,y$, and consequently the number of edges from $x$ to other vertices has distribution $\Po{\la\sum_{h>0}( b -1) b ^{h-1}\zeta_h}$. Write $\Xi_0$ for $\sum_{h>0}( b -1) b ^{h-1}\zeta_h$. Similarly, if we fix a vertex $v\in T_{\infty}$ with height $k$, write $\Xi_k$ for the value such that the number of edges in $\gnl[\infty]$ between descendants of $v$ and the rest of the graph has distribution $\Po{\Xi_k\la}$. Each descendant of $v$ has $\Po{\la\sum_{h>k}( b -1) b ^{h-1}\zeta_h}$ edges to the rest of the graph, and there are $ b ^k$ descendants, so $\Xi_k=( b -1) b ^k\sum_{h>k} b ^{h-1}\zeta_h$.

We will need some bounds on the value of $\Xi_k$.
\begin{lem}\label{twothirds}$\Xi_k$ is strictly decreasing with $\lim_k \Xi_k =  ( b -1)/( b +1)$. Moreover,
\[\frac{ b -1}{ b +1}\biggl(1+\frac{ b ^{-k}}{ b +1}\biggr)<\Xi_k<\frac{ b -1}{ b +1}\Bigl(1+\frac{1}{ b ^{k+1}-1}\Bigr)\Bigl(1+\frac{1}{ b ^{k+2}-1}\Bigr)\,.\]\end{lem}
\begin{proof}
By \Lr{zeta} we have $ b ^{i+1}\zeta_{i+1}< b ^i\zeta_i/ b $ for every $i$, and so $\sum_{h>k+1} b ^h\zeta_h<\sum_{h>k} b ^h\zeta_h/ b $, \ie $\Xi_{k+1}/(( b -1) b ^{k+1})<\Xi_k/(( b -1) b ^{k+1})$. So $\Xi_k$ is strictly decreasing. Also, by \Lr{zetabounds} we have
\begin{align*}
\Xi_k&=( b -1) b ^k\sum_{h>k} b ^{h-1}\zeta_h\\
&<( b -1) b ^k\sum_{h>k}\Bigl(1+\frac{1}{ b ^{h}-1}\Bigr)\Bigl(1+\frac{1}{ b ^{h+1}-1}\Bigr)\frac{ b -1}{ b +1} b ^{-h}\,.
\end{align*}
We can easily calculate $\sum_{h>k}\frac{b-1}{b+1}b^{-h}$; the other factors are close to $1$, so we bound them by their maximum values, which occur at $h=k+1$.
\begin{align*}
\Xi_k&<( b -1) b ^k\sum_{h>k}\Bigl(1+\frac{1}{ b ^{k+1}-1}\Bigr)\Bigl(1+\frac{1}{ b ^{k+2}-1}\Bigr)\frac{ b -1}{ b +1} b ^{-h}\\
&=\frac{ b -1}{ b +1}\Bigl(1+\frac{1}{ b ^{k+1}-1}\Bigr)\Bigl(1+\frac{1}{ b ^{k+2}-1}\Bigr)\,.
\end{align*}
Finally, using the lower bound from \Lr{zetabounds} we have
\begin{align*}
\Xi_k&>( b -1) b ^k\bfrac{ b -1}{ b +1}\sum_{h>k}( b ^{-h}+ b ^{-2h})\\
&=\frac{ b -1}{ b +1}\biggl(1+\frac{ b ^{-k}}{ b +1}\biggr)\,.
\end{align*}
Since these upper and lower bounds tend to $( b -1)/( b +1)$, so does $\Xi_k$.
\end{proof}

\begin{rem}The upper bound gives $\Xi_k<\frac{ b -1}{ b +1}\frac{ b }{ b -1}\frac{ b ^2}{ b ^2-1}=\frac{ b ^3}{ b ^3+ b ^2- b -1}<1$ for all $k\geq 0$. Also, since $\frac{ b ( b -1)}{( b +1)^2}$ is increasing and $ b \geq 2$, the lower bound gives $\Xi_k>\frac{ b -1}{ b +1}+\frac 29 b ^{-k-1}$.
\end{rem}

Finally, we prove bounds on the corresponding parameters for the finite model. Write $\hier{\xi}{k}{n}$ for the value such that the number of edges in $\gnl$ from the descendants of a vertex at height $k$ to other vertices is distributed $\Po{\la\hier{\xi}{k}{n}}$, \ie $\hier{\xi}{k}{n}$ is $b ^k$ times the probability that a random walk starting from a descendant of $v$ ends at a non-descendant.
\begin{lem}\label{xicalc}\[\Xi_k-b ^{2k-2n}<\hier{\xi}kn<\Xi_k\,.\]
\end{lem}
\begin{proof}We prove that \[\hier{\xi}{k}{n}=\Xi_k-( b -1)\sum_{h>n} b ^{h-n+2k-1}\zeta_h\,;\]
which clearly implies the upper bound. The lower bound follows since $\zeta_h< b ^{1-2h}$, and so
\[\sum_{h>n} b ^{h-n+2k-1}\zeta_h< b ^{2k-n}\sum_{h>n} b ^{-h}=\frac{ b ^{2k-2n}}{ b -1}\,.\]
We couple random walks on $T_n$ and $T_{\infty}$ as follows. Start by running the two walks identically. At any point when the walk on $T_{\infty}$ leaves $T_n$, pause the walk on $T_n$ at the apex until the walk on $T_{\infty}$ next reaches a vertex at height $n$. The descendants of this vertex form a copy of $T_n$; so long as the second walk stays within that copy, duplicate all its steps in the walk on $T_n$. Now the first random walk ends in $D(v)$ if and only if the second walk ends at a vertex in the translation of $D(v)$ to the subtree formed by the descendants of some vertex at height $n$. There are $( b -1) b ^{h-n-1+k}$ such vertices whose lowest common ancestor with $T_n$ is at height $h$, for every $h>n$, and so for any $x\in D(v)$
\[\prob[n]{x\wto D(v)}=\prob[\infty]{x\wto D(v)}+\sum_{h>n}( b -1) b ^{h-n-1+k}\zeta_h\,.\] Since $\hier{\xi}{k}{n}=b ^k(1-\prob[n]{x\wto D(v)})$ and $\Xi_k=b ^k(1-\prob[\infty]{x\wto D(v)})$, the result follows.
\end{proof}
In particular, $\hier{\xi}{k}{n}\to\Xi_k$. Also, provided $2n\geq 3k+5$, we may combine the bounds of \Lr{xicalc} and \Lr{twothirds} to get
\begin{align*}
\hier{\xi}{k}{n}&>\Xi_k-b ^{2k-2n}\\
&\geq\Xi_k-b ^{-k-5}\\
&>\frac{ b -1}{ b +1}+\frac 29 b ^{-k-1}-b ^{-k-5}\\
&>\frac{ b -1}{ b +1}+\frac 14 b ^{-k-1}\,.
\end{align*}

\section{Phase transitions for our finite models.} \label{secthresholds}

In this section we study the threshold $\la$ at which our \ppm $\gnl$ and \pem $G_n(\la)$ become connected for finite $n$. We will prove that there is a phase transition for the connectedness of $\gnl$ and $G_n(\la)$ occurring at $\la = \mc n$ where $\mc$ is a constant (different for the two models) depending on $b$ only. This phase transition occurs in a window of width logarithmic in $n$ (\Tr{logwindow} and \Tr{logwindowpem}). 

It turns out that for the \pem the asymptotic threshold for connectedness coincides with that for having no isolated vertices, just like in the \ER\ model \cite{ER}, while for the \ppm this is not the case (\Tr{diffthresh}). We therefore answer Problem~8.2 of \cite{gwrg} in the negative.

We start with two lemmas which will be needed for the analysis of each of our models. We then establish the connectedness threshold for the \ppm in \Sr{thresppm}, and the \pem is treated in \Sr{threspem} in a similar fashion. 

The following lemma bounds the probability of isolated vertices in a wide range of random graph settings where different edges appear with different probability.

\begin{lem}\label{matching}Let $G$ be a random graph on vertex set $[n]$, where each edge $ij$ is independently present with probability $p_{ij}$ and absent with probability $q_{ij}=1-p_{ij}$. Write $I_i$ for the event that $i$ is isolated, and $N$ for the number of isolated vertices. If $\prob{I_i}=q$ for every $i$ then $\prob{N=0}<2/(2+nq)$.
\end{lem}
\begin{proof}Note that $\prob{I_i}=\prod_{j\neq i}q_{ij}$, so this product equals $q$ for each $i$. Also,
\begin{align*}
\mean{N^2}&=\sum_{i,j}\prob{I_i\wedge I_j}\\ 
&=\sum_i\Bigl(q+\sum_{j\neq i}\Bigl(q\prod_{k\neq i,j}q_{jk}\Bigr)\Bigr)\\
&=\sum_i\Bigl(q+\sum_{j\neq i}q^2q_{ij}^{-1}\Bigr)\,.
\end{align*}
Now $q_{ij}^{-1}\geq 1$ for each $j\neq i$, and $\prod_{j\neq i}q_{ij}^{-1}=q^{-1}$, so $\sum_{j\neq i}q_{ij}^{-1}$ is maximised when one of the terms is $q^{-1}$ and the others are all $1$. To see this, note that if $q_{ij}^{-1},q_{ik}^{-1}>1$ for distinct $j,k$ then we increase the sum by replacing $q_{ij}^{-1}+q_{ik}^{-1}$ by $1+q_{ij}^{-1}q_{ik}^{-1}$; by a sequence of operations of this form we continue increasing the sum until only one term exceeds $1$. So
\begin{align*}
\mean{N^2}&\leq n\bigl(q+q^2(q^{-1}+n-2)\bigr)\\
&=n(2q-2q^2+nq^2)\\
&<2nq+\mean{N}^2\,,
\end{align*}
and so $\var(N)<2nq$. Consequently, by Cantelli's inequality, 
\[\prob{N=0}\leq\var(N)/(\var(N)+\mean{N}^2)<2/(2+nq)\,.\qedhere\]
\end{proof}
\begin{rem}In particular, if $nq\to\infty$ as $n\to\infty$, $\prob{N=0}\to 0$. The converse is also true. Suppose $nq\leq c$ infinitely often. The event that an individual vertex meets some edges is an increasing event, and so, by Harris's inequality \cite{Har60}, all such events are positively correlated. So $\prob{N=0}\geq(1-q)^n\geq(1-c/n)^n$ infinitely often, and this bound approaches $\ee[-c]>0$. Trivially if $nq\to 0$ then $\prob{N=0}\to 1$; again the converse is true since a non-trivial lower bound on $nq$ gives a non-trivial upper bound on $2/(2+nq)$.\end{rem}
\begin{rem}Suppose that each vertex has an additional probability $r$ of being active, and we are interested in the number of active isolated vertices. Now the expected number is $nqr$ and a similar analysis gives a bound of $(1+r)/(1+r+nqr)\leq 2/(2+nqr)$.\end{rem}

We next prove a sufficient condition for connectedness which will be the key ingredient in establishing the supercritical region for each model.

First, we introduce some notation. Let $T$ be an arbitrary finite $b$-ary tree, 
and let $v,w$ be two vertices of $T$. We say $v$ is a \textit{$k$-uncle} of $w$ if a sibling of $v$ is an ancestor of $w$ at distance $k$ or less. We say $v$ and $w$ are \textit{$k$-cousins} if there exists $w'$ which is an ancestor of $w$ at distance $k$ or less such that $w'$ is a $k$-uncle of $v$ (note that this definition is symmetric, since some sibling $v'$ of $w'$ is a $k$-ancestor of $v$ and $k$-uncle of $w$). If $(\seq x{ b })$ and $(\seq y{ b })$ are two disjoint $ b $-tuples of siblings then we say $(\seq x{ b })$ and $(\seq y{ b })$ are $k$-cousins if $x_1$ and $y_1$ are $k$-cousins (which will also imply that all other pairs are $k$-cousins).

Now let $G$ be a graph whose vertex set is $L(T)$, the set of leaves of $T$. We say that two vertices $x,y\in V(T)$ are \textit{linked} by $G$ if there exist vertices $x',y'\in L_n$ such that $x$ is an ancestor of $x'$, $y$ is an ancestor of $y'$, and there is an edge (in $G$) from $x'$ to $y'$. If $X=\set x{ b }$ is a set of siblings in $T$, we say $X$ is \textit{strongly linked} by $G$ if the graph $H_X$ on vertex set $X$ with edges between linked pairs is connected, and \textit{weakly linked} if $H_X$ has two components and there exist $i,j$ such that $x_i$ and $x_j$ are in different components and some $z$ which is a $k$-uncle of $x_i,x_j$ such that the pairs $x_i,z$ and $x_j,z$ are both linked.

\begin{lem}\label{linkage}Suppose that $G$ has the following properties, for some fixed $k$:
\begin{enumerate}[(i)]\item every set of siblings in $T$ is either strongly linked or weakly linked by $G$;
\item for any two sets of siblings which are $k$-cousins, at least one of them is strongly linked by $G$;
\item any set of siblings within the top $k$ layers of $T$ are strongly linked by $G$.\end{enumerate}
Then $G$ is connected.\end{lem}
\begin{proof}Define the \textit{depth} of a vertex $v\in V(T)$ to be the distance from the apex of $T$, and the \textit{height} $h(v)$ to be the difference between the depth of $v$ and the maximum depth (note that this coincides with our earlier definition for $T=T_n$). We show by induction on $j$ that for every vertex $v\in V(T)$ with $h(v)=j$, all the leaves of $T$ which are descendants of $v$ are in the same component of $G$. This is trivial for $j=0$, and whenever $v$ is a leaf. Suppose it is true for $0,\ldots,j-1$, let $v$ be a non-leaf vertex at height $j$, and let $\seq x{ b }$ be its neighbours at height $j-1$. Every descendant of $v$ is a descendant of some $x_i$, and for each $i$ all descendants of $x_i$ are in one component. If $\seq x{ b }$ are strongly linked then there is an edge between a descendant of $x_i$ and $x_j$ for $x_ix_j\in E(H_X)$, and these connect all the components so all descendants of $v$ are in the same component. If not, the descendants of $v$ are in at most two components, and there is some $z$ which is linked to both components and is a $k$-uncle of $\seq x{ b }$. It is sufficient to prove that all descendants of $z$ are in one component. If $h(z)<j$ then this is true by the induction hypothesis. If $h(z)\geq j$ we use a subsidiary induction to show that all descendants of $z'$ are in the same component whenever $z'$ is a descendant of $z$ at height $i$; this is true for $i=j-1$, and may be extended to each successive $i$ by noting that if $\seq{z'}{ b }$ are siblings which are descendants of $z$ at height at least $j-1$ then $(\seq x{ b })$ and $(\seq{z'}{ b })$ are $k$-cousins, so $(\seq{z'}{ b })$ are strongly linked by (ii).
\end{proof}

\subsection{Connectedness threshold for the \ppm}\label{thresppm}
We will show that, unlike the Erd\H{o}s--R\'enyi random graph model, isolated vertices are not the main obstacle to connectedness in $\gnl$.  

The bounds of \Sr{calc} will allow us to give, for a fixed height $k$, a threshold dividing regimes where there is a component of the form $D(v)$ for some vertex $v$ at height $k$ with high probability from regimes where there is no such component with high probability. In fact, of all sets of vertices with lowest common ancestor $v$, $D(v)$ is the most likely to be separated from the rest of the graph, and so we can show that in the latter regime there is actually no component with a common ancestor of height at most $k$ with high probability. There may be such a component even if other descendants of the common ancestor are not disconnected from the rest of the graph, so this is a genuinely stronger statement.
\begin{lem}\label{inout}Let $x\in L_k$ and $\varnothing\neq A\subset L_k$. If $k\leq m<n\leq\infty$ then $\prob[m]{x\wto A}>\prob[n]{x\wto A}$.\end{lem}
\begin{proof}We couple the random walks on $T_m$ and $T_n$ as in the proof of \Lr{xicalc}: duplicate every step made in the bottom $m$ layers of $T_n$, but pause the walk in $T_m$ while the walk in $T_n$ is above that level. The walk in $T_m$ reaches $A$ if and only if the walk in $T_n$ reaches the translation of $A$ to the $m$-subtree it finishes in.\end{proof}
\begin{lem}\label{segment}Let $A$ be any nonempty subset of $L_k$. Then the probability that there are no edges between $A$ and the rest of $\gnl$ is maximised when $A=L_k$, for any $n\geq k$ or for $n=\infty$.\end{lem}
\begin{proof}
We use induction on $k$; it is clearly true for $k=0$ since there is no choice of $A$. Assume it is true for $k-1$. Note that 
\[\prob{e(A,V\setminus A)=0}=\exp\Bigl(-\la\sum_{x\in A}\prob{x\wto V\setminus A}\Bigr)\,,\]
since $e(A,V\setminus A)$ has a Poisson distribution with the appropriate mean. Write $P(A)$ for $\sum_{x\in A}\prob{x\wto V\setminus A}$. Now
\[P(A)-P(L_k)=\sum_{x\in L_k\setminus A}\Bigl(\Bigl(\sum_{y\in A}\prob{y\wto x}\Bigr)-\prob{x\wto V\setminus L_k}\Bigr)\,.\]
By \Lr{inout}, $\prob{y\wto x}$ is minimised and $\prob{x\wto V\setminus L_k}$ maximised in the case $n=\infty$, so it is sufficient to prove $P(A)-P(L_k)>0$ in this case.
\begin{align*}
P(A)-P(L_k)&\geq\sum_{x\in L_k\setminus A}\Bigl(\abs{A}\zeta_k-\sum_{y\not\in L_k}\prob{x\wto y}\Bigr)\\
&=\sum_{x\in L_k\setminus A}\Bigl(\abs{A}\zeta_k-\sum_{h>k}( b -1) b ^{h-1}\zeta_h\Bigr)\\
&>\sum_{x\in L_k\setminus A}\Bigl(\abs{A}\zeta_k-\sum_{h>k}( b -1) b ^{h-1} b ^{2k-2h}\zeta_k\Bigr)\\
&=( b ^k-\abs{A})(\abs{A}- b ^{k-1})\zeta_k\,,
\end{align*}
so the result follows for all $A$ with $\abs{A}\geq b ^{k-1}$. If $\abs{A}< b ^{k-1}$, we claim that $\min_{\abs{B}=\abs{A}}P(B)$ is achieved when $B$ consists of the first $\abs{A}$ elements of $L_k$. Then, since $B\subset L_{k-1}$, we have $P(A)\geq P(B)\geq P(L_{k-1})$ by induction, and $P(L_{k-1})>P(L_k)$ since $\abs{L_{k-1}}= b ^{k-1}$.

To prove the claim, note that 
\begin{align*}
P(B)&=\abs{B}-\sum_{x\in B}\prob{x\wto B}\\
&=\abs{B}-\sum_{x,y\in B}\prob{x\wto y}\\
&=\abs{B}-\sum_{x,y\in B}\zeta_{h(x,y)}\,,
\end{align*}
where $h(x,y)$ is the height of the lowest common ancestor of $x$ and $y$. Since $\zeta_h$ is decreasing, it is sufficient to prove that an initial segment maximises $\abs{\{(x,y)\in B^2:h(x,y)\leq h\}}$ for each $h$. Separate $L_k$ into chunks of length $ b ^h$, and write $\seq br$ for the number of vertices of $b$ in each chunk. Then $\abs{\{(x,y)\in B^2:h(x,y)\leq h\}}=\sum_ib_i^2$. If $ b ^h>b_i\geq b_j>0$ then this can be increased by replacing $b_i,b_j$ with $b_i+1,b_j-1$. Consequently $P(B)$ is minimal only if at most one chunk is neither full nor empty for every $h$. There is a unique (up to reordering of the chunks) sequence $\seq br$ which achieves this, so any $B$ which achieves this for every $h$ has the same value of $P(B)$, which is minimal; in particular the initial segment is one such $B$.\end{proof}

Write $D_k$ for the event that there is some $v\in T_n$ at height $k$ such that the descendants of $v$ are disconnected from the rest of $\gnl$. Write $C_k$ for the event that $\gnl$ has some component with a common ancestor at height $k$ (not necessarily the lowest common ancestor, so $C_k\subset C_{k+1}$).
\begin{thm}\label{diffthresh}There exist values $\sg_0<\sg_1<\cdots$ with $\lim_{k\to\infty}\sg_k=\mc=\frac{b+1}{b-1}\log b$ such that
\begin{enumerate}[(a)]
\item $\la=\sg_kn$ is a sharp threshold for both $C_k$ and $D_k$, and
\item $\la=\mc n$ is a sharp threshold for connectedness of $\gnl$.\end{enumerate}
Further, if $\la=\sg_kn$ then $\prob{D_k}$ and $\prob{C_k}$ are bounded away from $0$ and $1$.\end{thm}
\begin{rem}In particular the sharp thresholds for the existence of isolated vertices ($\la=\sg_0n$) and for connectedness ($\la=\mc n$) do not coincide.\end{rem}
\begin{proof}
For (a), since $D_k\subseteq C_k$, it is sufficient to show that, for $\la=\sg n$, $D_k$ holds with high probability if $\sg<\sg_k$ and $C_k$ fails with high probability if $\sg>\sg_k$. 

Write $\hier Gjn(\la)$ to be the graph whose vertices are the vertices at height $j$ in $T_n$, with the number of edges between $v$ and $w$ in $\hier Gjn(\la)$ being equal to the number of edges between descendants of $v$ and descendants of $w$ in $\gnl$. Then $\hier G0n(\la)=\gnl$, and $v$ is isolated in $\hier G{h(v)}n(\la)$ if and only if its descendants are disconnected from the rest of $\gnl$. Note that adjacencies in $\hier Gjn(\la)$ are independent events with varying probabilities.

Suppose $\sg\Xi_k>\log b $. For $n$ sufficiently large, also $\sg\hier{\xi}kn>\log b $. The probability that a given vertex $v$ of height $h\leq k$ will be isolated 
in $\hier Ghn(\sg n)$ is $\exp(-\hier{\xi}hn\sg n)\leq\exp(-\hier{\xi}kn\sg n)=o( b ^{-n})$. By \Lr{segment}, each possible subset of $D(v)$ is a component of $\gnm{\sg n}$ 
with probability $o(2^{-n})$. There are less than $ b ^{n+1}$ such vertices $v$, and at most $2^{ b ^k}=O(1)$ subsets of $D(v)$ in each case, so $\prob{C_k}=o(1)$.

Conversely, suppose $\sg\Xi_k<\log b $. The probability that a given vertex $v$ at height $k$ will be isolated in $\hier Gkn(\sg n)$ is $\exp(-\hier{\xi}kn\sg n)$. 
But $\sg\hier{\xi}kn<\log b $, and so this probability is $\omega( b ^{-n})$. There are $ b ^{n-k}$ vertices in $\hier Gkn(\sg n)$, and so the expected number of isolated 
vertices is $\omega(1)$. By \Lr{matching}, $\prob{D_k}=1-o(1)$. Thus (a) holds for $\sg_k=\Xi_k^{-1}\log b $.

Now suppose $\sg=\sg_k=\Xi_k^{-1}\log b $. Write $N_{D_k}$ for the number of vertices at height $k$ which are isolated in $\hier Gkn(\sg n)$. Then
\begin{align*}
\mean{N_{D_k}}&= b ^{n-k}\exp(-\hier{\xi}kn\sg n)\\
&> b ^{n-k}\exp(-(\Xi_k-b ^{2k-2n})\sg n)\\
&= b ^{-k}\exp(b ^{2k-2n}\sg n)\\
&= b ^{-k}(1+o(1))\,.
\end{align*}
Consequently, by the remarks following \Lr{matching}, $\prob{D_k}$ is bounded away from $0$. 

For each nonempty set $S\subseteq L_k$, write $N_S$ for the number of clones of $S$ which are disconnected from the rest of the graph. By \Lr{segment}, for each such $S$ we have $\mean{N_S}\geq\mean{N_{D_k}}= b ^{-k}(1+o(1))$. Also, we can consider this as counting isolated, active vertices in some graph -- draw a graph whose vertices are the vertices at height $k$ in $T_n$, with an edge between two vertices if there is an edge in $\gnm{\sg n}$ between the clones of $S$ descended from them, and a vertex being active if there is no edge from the clone of $S$ descended from it to any vertex which is not part of one of the clones of $S$. Thus, by the remarks following \Lr{matching}, since $\mean{N_S}$ is bounded, we have $\prob{N_S=0}>\eps>0$ for some $\eps$. Now the events $(N_S=0)_{S\subseteq L_k}$ are all decreasing events, so positively correlated by Harris's inequality. Thus $\prob{N_S=0\text{ for every }S}>\eps^{2^{ b ^k}}$, and so $\prob{C_k}$ is bounded away from $1$. Since $\prob{C_k}>\prob{D_k}$, both bounds apply to both events, as required.

Since $\Xi_k\to \frac{b-1}{b+1}$ from above, we have $\mc=\lim_{k\to\infty}\sg_k=\frac{b+1}{b-1}\log b$. If $\sg<\mc$ then for some $k$ we have $\sg<\sg_k$, and so $\gnm{\sg n}$ is disconnected with high probability. It remains to show that $\gnm{\sg n}$ is connected with high probability for any $\sg>\mc$.

Fix $\sg>\mc$ and let $G=\gnm{\sg n}$ and $T=T_n$. It suffices to show that there is some fixed $k$ for which $G$ satisfies (i), (ii) and (iii) of \Lr{linkage} with high probability. Let $\eps>0$ be such that $\sg=(1+\eps)\mc$.

First we show that with high probability every group of siblings which is not strongly linked has one sibling which is not linked to any of the others, but all other pairs linked. It is sufficient to show that with high probability there is no group of siblings containing $ b $ unlinked pairs. 

Fix siblings $\seq x b $ at height $h-1$; for any pair $x_i', x_j'$ of descendants of $x_i,x_j$ where $i\neq j$, the number of edges between $x_i'$ and $x_j'$ is $\Po{\zeta_h\sg n}$. Consequently the total number of such edges, over all possible pairs $(x_i',x_j')$, is $\Po{b ^{2h-2}\zeta_h\sg n}$. Recalling our bounds on $\zeta_h$ from \Lr{zeta}, $b^{2h-1}\zeta_h>\frac{b-1}{b+1}$. Consequently, the probability that $x_i$ and $x_j$ are not linked is $\exp(-b ^{2h-2}\zeta_h\sg n)<\exp\bigl(-\frac{b-1}{b(b+1)}\sg n\bigr)$. Therefore, writing $N_{\seq xb}$ for the number of pairs $i\neq j$ such that $x_i$ and $x_j$ are not linked,
\begin{align*}\prob{N_{\seq xb}\geq b}&\leq\binom{\binom b2}b\exp\biggl(-\frac{b-1}{b+1}\sg n\biggr)\\
&=O(b^{-(1+\eps)n})\,,\end{align*}
and so the probability that this holds true for some set of siblings is $O(b^{-\eps n})=o(1)$.

Similarly, the probability that any given set of siblings has one sibling which is not linked to any of the others, but all other pairs linked, is at most $b\exp\bigl(-\frac{(b-1)^2}{b(b+1)}\sg n\bigr)$. This probability is $o(1)$, and there are only a constant number of groups of siblings at height $n-k$ or above, so certainly (iii) is satisfied with high probability. 

Suppose $\seq xb$ are siblings below this point, and write $\set{z^{(1)}}{b-1},\ldots,\set{z^{(k)}}{b-1}$ for the groups of $k$-uncles of $\seq xb$ (in increasing order of height). The number of edges which link $\hier zji$ to $x_m$ is $\Po{b^{2h+j-2}\zeta_{h+j-1}\sg n}$ and so 
\begin{align*}\prob{x_m,\hier zji\text{ linked}}&=1-\exp(-b^{2h+j-2}\zeta_{h+j-1}\sg n)\\
&>1-\exp\biggl(-\frac{b-1}{b^{j+1}(b+1)}\sg n\biggr)\,.\end{align*}
Consequently, since edges between different pairs of vertices are independent, 
\begin{align*}\prob{\exists m:x_m,\hier zji\text{ not linked}}&\leq1-\biggl(1-\exp\biggl(-\frac{b-1}{b^{j+1}(b+1)}\sg n\biggr)\biggr)^b\\
&<b\exp\biggl(-\frac{b-1}{b^{j+1}(b+1)}\sg n\biggr)\,.\end{align*}
The event $X_{\seq xb}$ that $\seq xb$ are neither strongly linked nor weakly linked therefore satisfies 
\begin{align*}\prob{X_{\seq xb}}&<b\exp\biggl(-\frac{(b-1)^2}{b(b+1)}\sg n\biggr)\prod_{j=1}^k\prod_{i=1}^{b-1}b\exp\biggl(-\frac{(b-1)}{b^{j+1}(b+1)}\sg n\biggr)\\
&=b\exp\biggl(-\frac{(b-1)^2}{b(b+1)}\sg n\biggr)\prod_{j=1}^kb^{b-1}\exp\biggl(-\frac{(b-1)^2}{b^{j+1}(b+1)}\sg n\biggr)\\
&=b^{kb-k+1}\exp\biggl(-\sum_{j=0}^k\frac{(b-1)^2}{b^{j+1}(b+1)}\sg n\biggr)\\
&=b^{kb-k+1}\exp\biggl(-(1-b^{-k-1})\frac{(b-1)}{b+1}\sg n\biggr)\\
&=b^{kb-k+1}\exp(-(1-b^{-k-1})(1+\eps)n\log b)\,.\end{align*}
For some value of $k$ which depends only on $\eps$ we have $\prob{X_{\seq xb}}=O(b^{(1+\eps/2)n})$. Thus the expected number of such events which occur is $O(b^{\eps n/2})$ \ie the probability that (i) fails to be satisfied is $o(1)$. 

Likewise, if the groups of siblings $\seq xb$ and $\seq yb$ are $k$-cousins, the probability that both groups have a single vertex not linked to any of the others is at most $b^2\exp\bigl(-\frac{(b-1)^2}{b(b+1)}\sg n\bigr)$ (since these are independent events), and for any given $\seq xb$ there are fewer than $k(b-1)b^k$ possible choices for $\seq yb$, since any $k$-cousins of $\seq xb$ are in a copy of $T_{k}$ with apex a $k$-uncle of $\seq xb$. Consequently the probability that (ii) fails is at most $k(b-1)b^{n+k+2}\exp\bigl(-\frac{(b-1)^2}{b(b+1)}\sg n\bigr)$. Since $\frac{(b-1)^2}{b(b+1)}\sg\geq(1+\eps)\log b$, this probability is $o(1)$, as required.

Thus the conditions of \Lr{linkage} hold with high probability for this choice of $k$, and so (b) holds.
\end{proof}

We can improve the results of \Tr{diffthresh} to give a logarithmic window in which connectedness occurs.
\begin{thm} \label{logwindow}
For any fixed $\alpha>b+1$, $\gnm{\mc n+\alpha\log n}$ is connected with high probability and for any fixed $\beta>\frac{b+1}{b-1}$, $\gnm{\mc n-\beta\log n}$ is disconnected with high probability.\end{thm}
\begin{proof}We follow the proof of \Tr{diffthresh}, but instead of taking $k$ constant, set $k=k_n=\log_bn$ (note that this logarithm is to base $b$ but the $\log n$ in the statement of the theorem is natural). This is the right multiple to take: if $k=\gamma\log_bn$ then $\gamma<1$ doesn't work and $\gamma>1$ forces $\alpha,\beta$ to be higher.

For $\gnm{\mc n+\alpha\log n}$, we need to show that with high probability (i), (ii) and (iii) still apply. As before, for any group of siblings $\seq xb$,
\begin{align*}\prob{N_{\seq xb}\geq b}&\leq\binom{\binom b2}b\exp\biggl(-\frac{b-1}{b+1}(\mc n+\alpha\log n)\biggr)\\
&=O(b^{-n}n^{-\alpha(b-1)/(b+1)})\,,\end{align*}
and so with high probability no group of siblings fails to be strongly connected except by having one vertex not linked to any of the others. Also,
\begin{align*}
\prob{X_{\seq xb}}&<b^{kb-k+1}\exp\biggl(-(1-b^{-k-1})\frac{b-1}{b+1}(\mc n+\alpha\log n)\biggr)\\
&=bn^{b-1}b^{-n}n^{-\alpha(b-1)/(b+1)}b^{1/b}n^{o(1)}\\
&=b^{b+1-n}n^{(b-1)(1+o(1)-\alpha/(b+1))}\,.
\end{align*}
As $\alpha>b+1$, this is $o(b^{-n})$, and so with high probability (i) holds for every group $\seq xb$.

As before, the probability of a particular group of siblings not being strongly linked is at most $b\exp\bigl(-\frac{(b-1)^2}{b(b+1)}(\mc n+\alpha\log n)\bigr)<b^{-(b-1)n/b}$. Since there are fewer than $b^k=n$ pairs of siblings in the top $k$ layers, the probability that some such pair are not linked is at most $nb^{-(b-1)n/b}=o(1)$, so (iii) holds with high probability.

Since the events of different groups of siblings being strongly linked are independent, for any groups of siblings $\seq xb$ and $\seq yb$ which are $k$-cousins, the probability that both groups have a vertex not linked to any of the others is at most $b^2\exp\bigl(-\frac{(b-1)^2}{b(b+1)}(\mc n+\alpha\log n)\bigr)\leq b^{-n}n^{-\alpha(b-1)^2/(b(b+1))}$. Since $\alpha>b+1$ and $b\geq 2$, we have $\alpha(b-1)^2/(b(b+1))-1=\delta>0$. There are at most $b^n$ choices for $\seq xb$, and for any such choice at most $k(b-1)b^k$ possible choices of $\seq yb$, so the probability that (ii) fails is $O(n^{-\delta}\log n)=o(1)$, as required.

For $\gnm{\mc n-\beta\log n}$, we show that with high probability some vertex of height $k$ is isolated in $\hier Gkn(\mc n-\beta\log n)$. Write $I_v$ for the event that a given vertex $v$ is isolated. Noting that $\hier{\xi}kn<\Xi_k<\frac{b-1}{b+1}\bigl(1+\frac{b^{-k}}{b+1}\bigr)$,
\begin{align*}
\prob{I_v}&=\exp(-\hier{\xi}kn(\mc n-\beta\log n))\\
&>\exp\biggl(-\frac{b-1}{b+1}\biggl(1+\frac{n^{-1}}{b+1}\biggr)(\mc n-\beta\log n)\biggr)\\
&=\exp\biggl(-n\log b+\frac{\beta(b-1)}{b+1}\log n-\frac{\log b}{b+1}+O(n^{-1}\log n)\biggr)\\
&=\Theta(b^{-n}n^{\beta(b-1)/(b+1)})\,.
\end{align*}
Since this is $\omega(b^{k-n})$, by \Lr{matching} there is at least one such $v$ with high probability.
\end{proof}

\subsection{Connectedness threshold for the \pem} \label{threspem}
In this section we consider the connectedness threshold for the \pem restricted to $T_n$. This gives a random graph $G_n(\la)$. Now write $\hier{\hat{\xi}}{k}{n}$ for the value such that the number of edges from the descendants of a vertex at height $k$ to other vertices has distribution $\Po{\la\hier{\hat{\xi}}{k}{n}}$. Now
\begin{align*}\hier{\hat{\xi}}{k}{n}&=\sum_{h=k+1}^{n}b^k(b-1)b^{h-1}b^{1-2h}\\
&=(1-b^{k-n})\,.\end{align*}
Consequently $\hat{\Xi}_k=\lim_{n\to\infty}\hier{\hat{\xi}}{k}{n}=1$ for each $k$, so all the thresholds seen in the previous section coincide.
Adapting \Tr{logwindow} to this setting gives the following result, saying in particular that $\la=n\log b$ is a tight threshold for connectivity.
\begin{thm}\label{logwindowpem}Set $\la=n\log b+f(n)$. Then $G_n(\la)$ has isolated vertices with high probability if $f(n)\to-\infty$, and with probability bounded away from $0$ and $1$ if $f(n)=O(1)$, whereas if $f(n)\geq\alpha\log n$ then $G_n(\la)$ is connected with high probability.\end{thm}
\begin{proof}Write $X$ for the expected number of isolated vertices. Then 
\begin{align*}
\mean{X}&=b^n\exp(-(1-b^{-n})\la)\\
&=b^{nb^{-n}}\exp(-(1-b^{-n})f(n))\\
&=(1+o(1))\exp(-(1-b^{-n})f(n))\,,
\end{align*}
so $(1+o(1))\ee[-f(n)]<\mean{X}<(1+o(1))\ee[-(1-b^{-1})f(n)]$. Therefore $\mean{X}\to 0$ if $f(n)\to-\infty$, and $\mean{X}=\Theta(1)$ if $f(n)=O(1)$. In either case, by \Lr{matching} we have the required result.

It remains to show that $G_n(n\log b+\alpha\log n)$ is connected with high probability. We follow the same approach as in proving \Tr{logwindow}: set $k=\log_bn$ and show that conditions (i), (ii) and (iii) of \Lr{linkage} are satisfied. In the \pem, the probability that siblings $x_i,x_j$ at height $h-1$ are not linked is $\ee[-\la/b]$ and hence for any set of siblings $\seq xb$ we have $\prob{N_{\seq xb}\geq b}=O(\ee[-\la])=O(b^nn^{\alpha})=o(b^n)$, so with high probability this does not occur for any set of siblings. The probability that $\seq xb$ has some vertex not linked to any of the others is at most $b\ee[-\la(b-1)/b]=o(n^{-1})$, so with high probability none of the $b^k=n$ sets in the top $k$ layers fail to be strongly linked, and (iii) holds.

If $\seq xb$ are siblings below this point, by a similar calculation to that in \Tr{logwindow} we get
\begin{align*}\prob{X_{\seq xb}}&<b^{kb-k+1}\exp\bigl(-(1-b^{-k-1})(n\log b+\alpha\log n)\bigr)\\
&=bn^{b-1}b^{-n}n^{-\alpha}b^{1/b}n^{o(1)}\,.\end{align*}
Since $\alpha>b-1$,  this probability is $o(b^{-n})$. Consequently with high probability this does not occur for any set of siblings, so (i) holds. 

Since the events of different groups of siblings being strongly linked are independent, for any groups of siblings $\seq xb$ and $\seq yb$ which are $k$-cousins, the probability that both groups have a vertex not linked to any of the others is at most $b^2\ee[-2\la(b-1)/b]=b^2(b^{-2n}n^{-2\alpha})^{(b-1)/b}$. There are at most $b^n$ choices for $\seq xb$, and for any such choice at most $k(b-1)b^k=(b-1)n\log_bn$ possible choices of $\seq yb$. Since $(b-1)/b\geq 1/2$ and $\alpha>1$, we have that the expected number of pairs $(\seq xb),\seq yb$ which are $k$-cousins and neither of which are strongly linked is $o(1)$. Consequently (ii) also holds with high probability.
\end{proof}

\section{Open problems}
We have proved (\Tr{Ebounds}) an exponential lower bound and a doubly exponential upper bound on the expected size $\chi(\la)$ of the random graph $G(\la)$ of \Prr{thminvariant}. Simulations suggest that $\chi(\la)$ might be of order $\la^{c\la}$. We would be very interested in any progress:
\begin{problem}
Determine the growth of $\chi(\la)$, or provide better bounds.
\end{problem}

We conclude with some rather general problems for percolation on groups.
A well-known conjecture of Benjamini \& Schramm \cite{BeSchrPer} states that every Cayley graph of a group which is not virtually $\z$ satisfies $p_{\mathrm c}<1$. A proof by Duminil-Copin, Goswami, Raoufi, Severo, and Yadin \cite{DCGRSY} appeared while these lines were being written. In a similar spirit, we ask

\begin{problem} \label{probperc}
Does every countable group admit a generating measure $\mu$ with $\la_{\mathrm c}(\mu) < \infty$?
\end{problem}
Here $\mu$ and $\la_{\mathrm c}$ are as introduced in \Sr{secintromu}. Note that we do not have to make an exception for groups that are virtually $\z$ this time, because, as shown by known results on long range percolation (see \Sr{secintromu}), $\z$ does admit such a $\mu$. Thus \Prb{probperc} is only open for groups having no infinite finitely generated subgroup.

Another conjecture of \cite{BeSchrPer} states that almost transitive graphs with isoperimetric dimension greater than 1 satisfy $p_{\mathrm c}<1$, where the \defi{isoperimetric dimension} of a graph $G$ is defined as 
\begin{equation}\dim(G):= \sup \Bigl\{ d>0 \Bigm| \inf_{S \Subset V(G)}  \frac{\abs{\partial S}}{\abs{S}^{1-1/d}} >0 \Bigr\}.\label{isodim}\end{equation}
Here, $S \Subset V(G)$ means that $S$ is a finite set of vertices of $G$, and $\abs{\partial S}$ denotes the number of edges with exactly one endvertex in $S$. In our set-up, we can define the \defi{isoperimetric dimension} $\dim(\mu)$ of a generating measure $\mu$ similarly, by letting $\abs{\partial S}$ denote the total measure of the edges with exactly one endvertex in $S$, \ie by letting $\abs{\partial S}:= \sum_{g\in S, h\not\in S} \mu(g^{-1}h)$, and otherwise leaving \eqref{isodim} unchanged. We can then ask 
\begin{problem}
Does $\dim(\mu)>1$ imply $\la_{\mathrm c}(\mu) < \infty$? 
\end{problem}

In fact, the example of long range percolation suggests that a much weaker isoperimetric condition might be sufficient for percolation (cf.~\cite[Conjecture 3]{BeSchrPer}):
\begin{problem}
Let $\mu$ be generating measure on a group $G$ such that $\abs{\partial S}> c \log\abs{S}$ holds for some constant $c>0$ and every $S \Subset V(G)$. Must $\la_{\mathrm c}(\mu) < \infty$ hold? 
\end{problem}

We can define the \defi{percolation threshold} of a group $G$ by 
\[\la_{\mathrm c}(G):= \inf_{\mu} \la_{\mathrm c}(\mu),\]
where the infimum ranges over all generating measures $\mu$  on $G$. It is not hard to see that $\la_{\mathrm c}(G) \geq 1$ by comparing with the Poisson branching process. Some nice features of $\la_{\mathrm c}(G)$ are that it is a group invariant, monotone with respect to the subgroup and the quotient group relations. Yet, it is  unclear whether this is a trivial concept:
\begin{problem}
Is there a group $G$ with $\la_{\mathrm c}(G)\neq 1$?
\end{problem}
This problem is the opposite extreme of \Prb{probperc}, which asks whether $\la_{\mathrm c}(G)<\infty$ \fe\ $G$.

A generating measure $\mu$  also naturally defines a random walk on $\Gamma$, by letting the transition probability from an element $g$ to an element $h$ be $\mu(g^{-1}h)$. Identifying properties of this random walk that are determined by $\Gamma$ and are independent from the choice of $\mu$ is an interesting and widely studied topic \cite{kaimanovich_random_1983,woessBook}. It would be interesting to compare the behaviour of the random walk to that of our percolation model for the same $\mu$. In fact, the same $\mu$ can be used to define to and compare with other models of statistical mechanics, \eg first passage percolation, and one can consider corresponding group invariants analogous to $\la_{\mathrm c}(G)$.  We hope to pursue these ideas in future research.

\section*{Acknowledgements}
We thank Omer Angel and Gourab Ray for the important observations mentioned above.


\begin{thebibliography}{10}
\bibitem{AizNewTre}
Aizenman, M. and Newman, C.\,M. (1984)
Tree graph inequalities and critical behavior in percolation models.
\textit{J. Statist.\ Phys.}\ \textbf{36} 107--143.

\bibitem{AizNewDis}
Aizenman, M. and Newman, C.\,M. (1986)
Discontinuity of the percolation density in one-dimensional $1/\abs{x-y}^2$ percolation models.
\textit{Comm.\ Math.\ Phys.}\ \textbf{107} 611--647.

\bibitem{BeSchrPer}
Benjamini, I. and Schramm, O. (1996)
Percolation beyond $\mathbb{Z}^d$, many questions and a few answers.
\textit{Electron.\ Comm.\ Probab.}\ \textbf{1} 71--82.

\bibitem{BeSchrRec}
Benjamini, I. and Schramm, O. (2001)
Recurrence of distributional limits of finite planar graphs.
\textit{Electron.\ J. Probab.}\ \textbf{6} no 23.

\bibitem{DCGRSY}
Duminil-Copin, H., Goswami, S., Raoufi, A., Severo, F. and Yadin, A. (2018)
Existence of phase transition for percolation using the Gaussian Free Field.
Preprint, arXiv:1806.07733.

\bibitem{ER}
Erd\"os, P. and R\'enyi, A. (1959)
On random graphs I.
\textit{Publ.\ Math.\ Debrecen} \textbf{6} 290--297.

\bibitem{gwrg}
Georgakopoulos, A. (2016)
Group-Walk Random Graphs. 
In \textit{Groups, Graphs, and Random Walks} 
(T. Ceccherini-Silberstein, M. Salvatori and E. Sava-Huss, eds), 
Vol.\ 436 of the \textit{LMS Lecture Note Series}, Cambridge University Press, pp.\ 190--204.

\bibitem{analyticity}
Georgakopoulos, A. and Panagiotis, C. (2018)
Analyticity results in Bernoulli Percolation.
Preprint, arXiv:1811.07404.

\bibitem{Har60}
Harris, T.\,E. (1960)
A lower bound for the critical probability in a certain percolation process.
\textit{Proc.\ Cambridge Philos.\ Soc.}\ \textbf{56} 13--20.

\bibitem{HofNachUnl}
van der Hofstad, R. and Nachmias, A. (2014)
Unlacing hypercube percolation: a survey.
\textit{Metrika} \textbf{77} 23--50.

\bibitem{HofNachHyp}
van der Hofstad, R. and Nachmias, A. (2017)
Hypercube percolation.
\textit{J. Eur.\ Math.\ Soc.\ (JEMS)} \textbf{19} 725--814.

\bibitem{kaimanovich_random_1983}
Kaimanovich, V.\,A.  and Vershik, A.\,M. (1983) 
Random walks on discrete groups: boundary and entropy.
\textit{Ann.\ Probab.}\ \textbf{11} 457--490.

\bibitem{Ke81}
Kesten, H. (1981)
Analyticity properties and power law estimates of functions in percolation theory.
\textit{J. Statist.\ Phys.}\ \textbf{25} 717--756.

\bibitem{LoSzeRan}
Lov\'asz, L. and Szegedy, B. (2012)
Random graphons and a weak Positivstellensatz for graphs.
\textit{J. Graph Theory} \textbf{70} 214--225.

\bibitem{SchulLong}
Schulman, L.\,S. (1983)
Long range percolation in one dimension.
\textit{J. Phys.\ A.} \textbf{16} L639--L641.

\bibitem{NewSchulOne}
C.~M. Newman and L.~S. Schulman. (1986)
One dimensional $1/\abs{j-i}^s$ percolation models: the existence of a transition for $s \leq 2$.
\textit{Comm.\ Math.\ Phys.}\ \textbf{104} 547--571.

\bibitem{PenRgg}
Penrose, M. (2003)
\textit{Random Geometric Graphs},
Vol.\ 5 of \textit{Oxford Studies in Probability}, Oxford University Press.

\bibitem{SteiSur}
Steif, J.\,E. (2009) 
\newblock {A Survey of Dynamical Percolation}.
\newblock In \textit{Fractal Geometry and Stochastics IV} (C. Bandt, M. Z\"ahle and P. M\"orters, eds), 
Vol.\ \textbf{61} of \textit{Progress in Probability}, Birkh\"auser Verlag, pp.\ 145--174.

\bibitem{woessBook}
Woess, W. (2002)
\textit{Random Walks on Infinite Graphs and Groups},
Vol.\ 138 of \textit{Cambridge Tracts in Mathematics},
Cambridge University Press.

\end{thebibliography}
\end{document}